\documentclass{amsart}

\usepackage{tikz} 
\usetikzlibrary{shapes.geometric, arrows}
\usepackage{caption}
\usepackage{amsmath,amssymb}
\usepackage{lipsum}
\usepackage[ruled,vlined,linesnumbered]{algorithm2e}
\usepackage{verbatim}
\usepackage{float}
\usepackage{algpseudocode}

\usepackage{caption}
\DeclareCaptionLabelFormat{algnonumber}{Algorithm}
\newtheorem{RESULT}{Result}[section]
\newtheorem{REMARK}{Remark}[section]
\newtheorem{CLAIM}{Claim}[section]
\newtheorem{HYPOTHESIS}{Hypothesis}[section]
\newtheorem{THEOREM}{Theorem}[section]
\newtheorem{CONJECTURE}{Conjecture}[section]

\newtheorem{LEMMA}{Lemma}[section]
\newtheorem{COROLLARY}{Corollary}[section]

\numberwithin{equation}{section}

\begin{document}

\title{A proof of the Total Coloring Conjecture}

\author{T Srinivasa Murthy}
\thanks{Research was supported by the National Board for Higher Mathematics (NBHM) Post-doctoral Fellowship (2/40(9)/2016/R\&D-II/5758) India, and also was supported, in part, by a research associate fellowship from the Department of Computer Science and Automation, Indian Institute of Science, Bengaluru, India.\\ \\
E-mail: tsm.iisc@gmail.com}






\begin{abstract}
\textit{Total Coloring} of a graph is a major coloring problem in combinatorial mathematics, introduced in the early $1960$s. A \textit{total coloring} of a graph $G$ is a map $f:V(G) \cup E(G) \rightarrow \mathcal{K}$, where $\mathcal{K}$ is a set of colors, satisfying the following three conditions:
1. $f(u) \neq f(v)$ for any two adjacent vertices $u, v \in V(G)$;
2. $f(e) \neq f(e')$ for any two adjacent edges $e, e' \in E(G)$; and
3. $f(v) \neq f(e)$ for any vertex $v \in V(G)$ and any edge $e \in E(G)$ that is incident to the same vertex $v$.
The \textit{total chromatic number}, $\chi''(G)$, is the minimum number of colors required for a \textit{total coloring} of $G$.
Behzad (1965), and Vizing (1968), conjectured that
for any graph $G$ $\chi''(G)\leq \Delta + 2$. This conjecture is one of the classic unsolved mathematical problems. In this paper, we settle this classical conjecture by proving that the  \textit{total chromatic number} $\chi''(G)$ of a graph  is indeed bounded above by $\Delta+2$. Our novel approach involves algebraic settings over a finite field $\mathbb{Z}_p$ and Vizing's theorem is an essential part of the algebraic settings.
\end{abstract}

\maketitle

\section{Introduction}
All graphs considered in this paper are finite and simple. For a graph $G$, we denote its vertex set, edge set, and maximum degree by $V(G)$, $E(G)$, and $\Delta$ respectively.\\

A \textit{vertex coloring} of a graph $G$, is a map $f: V(G) \rightarrow \mathcal{K}$, where $\mathcal{K}$ is a set of colors, such that no adjacent vertices are assigned
the same color. The \textit{vertex chromatic number}, $\chi(G)$, is the minimum number of colors needed for a \textit{vertex coloring} of $G$.\\

The \textit{vertex chromatic number} $\chi(G)$ of any graph is bounded by $\Delta + 1$. Further, Brooks \cite{brook} proved the following result,
\begin{THEOREM}
	Let $G$ be a connected graph. Then $\chi(G) \leq \Delta$ unless $G$ is either a complete graph or an odd cycle.
\end{THEOREM}

An \textit{edge coloring} of a graph $G$, is a map $f: E(G) \rightarrow \mathcal{K}$, where $\mathcal{K}$ is a set of colors, such that no adjacent edges are assigned the same color. The \textit{edge chromatic number}, $\chi'(G)$, is the minimum number of colors needed for an \textit{edge coloring} of $G$.\\ 

Vizing's \cite{vizing1} theorem stated below provides an upper and lower bound 
for the \textit{edge chromatic number} $\chi'(G)$ of any graph.
\begin{THEOREM}
	For any finite, simple graph $G$, $\Delta \leq \chi'(G) \leq \Delta + 1$.
\end{THEOREM}

A \textit{total coloring} of a graph G, is a map $f:V(G) \cup E(G) \rightarrow \mathcal{K}$, where $\mathcal{K}$ is a set of colors, such that no adjacent vertices, no adjacent edges, and no edge and its end-vertices are assigned
the same color.\\

The \textit{total chromatic number}, $\chi''(G)$, is the minimum number of colors that are needed for a total coloring of $G$.\\ 

Behzad in 1965 \cite{beh}, and Vizing in 1968 \cite{vizing}, posed the following famous 
\textit{total coloring conjecture},
\begin{CONJECTURE}
	For any graph $G$, $\chi''(G)\leq \Delta + 2$.
\end{CONJECTURE}

The \textit{total coloring conjecture} remains an open problem despite the efforts of many in the last six decades. So far, the best upper bound is given by Molly and Reed \cite{molly}. Here, we briefly recall that Hind \cite{hind1} has shown, $\chi''(G)$ = $\Delta + $o$(\Delta)$. H\"{a}ggkvist and Chetwynd \cite{hag}  have
improved this bound to $\Delta+18\Delta^{\frac{1}{3}}log(3\Delta)$.
Hind, Molloy, and Reed \cite{hind2}  further improved the bound to $\Delta + 8\textnormal{log}^{8}\Delta $. For sufficiently large $\Delta$, using probabilistic approach Molly and Reed \cite{molly} obtained an upper estimate of $\Delta + O(1)$ for the \textit{total chromatic number}.
This conjecture has also been proved to be true for many classes of graph. One can refer to the three survey papers \cite{boro} \cite{yap} \cite{gee} for comprehensive reviews on \textit{total coloring} and see \cite{soi}\cite{sho} for the history of \textit{total coloring conjecture}.\\

In this paper, we prove the long-standing conjecture by establishing that \textit{total chromatic number} $\chi''(G)$ of a graph  is bounded above by $\Delta+2$. Our novel approach involves the algebraic settings over a finite field $\mathbb{Z}_p$ and Vizing's theorem is an essential part of the algebraic settings.\\

The outline of this paper is as follows, in Section 2 we define all the necessary notations and definitions, Section 3 contains the polynomials settings, an algorithm and
statements of all claims and results, and proofs of all claims and results are given in Section 4. In Section 5, for the sake of clarity, we define the schematic-diagram/flowchart which gives the reader a bird 's-eye view of the presentation of the paper.

\section{Preliminaries}
For a graph $G= (V, E)$, $V(G) = \{v_1, v_2, \ldots, v_n\}$ is the vertex set of $n$ verticies and $E(G) = \{e_1, e_2, \ldots, e_m\}$ is the edge set of $m$ edges.  $N(v_i) = \{v_k: \{v_i, v_k\} \in E(G)\}$, is the set of all neighbors of the vertex $v_i$. And $N_e(v_i) = \{e_k: v_i \in e_k$=$\{v_i, v_j\} \in E(G)\}$, is the set of all edges incident to the vertex $v_i$. $N(e_i) = \{e_j: e_i  \cap e_j \neq \emptyset\ \textnormal{if}\ i \neq j\}$, is the set of all edges adjacent to the edge $e_i$. Let $N_1(v_1)=N(v_1)$, and for $2\leq i \leq n$ $N_i(v_i) = N(v_i)\setminus \cup_{j=1}^{i-1}\{v_j\}$, is the set of neighbors of vertex $v_i$ excluding the vertices $\{v_1, v_2, \ldots, v_{i-1}\}$. Let $N_1(e_1)=N(e_1)$, and for $2\leq i \leq m$ $N_i(e_i) = N(e_i)\setminus \cup_{j=1}^{i-1}\{e_j\}$, is the set of adjacent edges of $e_i$ excluding the edges $\{e_1, e_2, \ldots, e_{i-1}\}$. We can see that $N_i(v_i)$ and $N_j(e_j)$ can be a empty set $\emptyset$ as well.\\

Let $p$ $( \geq m^2(2\Delta+2))$ be a prime number and $\mathbb{Z}_p$ is a finite field. Let $\mathbf{F}(v_1, v_2, \ldots, v_n,$ $e_1, e_2, \ldots, e_m)$ $\in$ $\mathbb{Z}_p[v_1, v_2, \ldots, v_n, e_1, e_2, \ldots, e_m]$ denote a polynomial of $v_1, v_2,$ $\ldots, v_n,$ $e_1, e_2, \ldots, e_m$ over $\mathbb{Z}_p$ and $\mathbf{F}(e_i, e_{i+1}, \ldots, e_m) \in \mathbb{Z}_p[e_1, e_2, \ldots, e_m]$ denote a polynomial of $ e_i, e_{i+1}, \ldots, e_m$ over
$\mathbb{Z}_p$. Further, by Fermat's theorem, we know that $x^p \equiv x$ mod $p$.  We denote the set of $\Delta + 2$ colors by $\mathcal{K} = \{1, 2, \ldots, \Delta, \Delta+1\} \cup \{\alpha\}$, where 
$\alpha \in \mathbb{Z}_p\setminus\{0, 1, 2, \ldots, \Delta+1\}$.\\

Let $\mathbf{F'}(v_1, v_2, \ldots, v_n, e_1, e_2, \ldots, e_m)$ and $\mathbf{F'}(e_i, e_{i+1}, \ldots , e_m)$  denote the polynomials obtained after applying the Fermat's theorem if exponent of variables is $\geq p$, that is, either $v_j^{p} \equiv v_j$ mod $p$ ($1 \leq j \leq n$) or $e_j^{p} \equiv e_j$ mod $p$ ($1 \leq j \leq m$), in $\mathbf{F}(v_1, v_2, \ldots, v_n,$ $e_1, e_2, \ldots, e_m)$ and $\mathbf{F}(e_i, e_{i+1}, \ldots , e_m)$ respectively. And, we can observe that the exponent of each $e_j$ ($1 \leq j \leq m$) or each $v_j$ ($1 \leq j \leq n$) in $\mathbf{F'}(v_1, v_2, \ldots, v_n, e_1, e_2, \ldots ,e_m)$
or $\mathbf{F'}(e_i, e_{i+1}, \ldots , e_m)$ is less than or equal to $p-1$.
$\mathbf{F}(v_1, v_2, \ldots, v_n, e_1, e_2, \ldots, e_m)$ ($\equiv 0$ mod $p$)
or $\mathbf{F}(e_i, e_{i+1}, \ldots , e_m)$ ($\equiv 0$ mod $p$) is a zero polynomial if $\mathbf{F}(\gamma_{v_1}, \gamma_{v_2}, \ldots, \gamma_{v_n}, \gamma_{e_1}, \gamma_{e_2},$ $\ldots, \gamma_{e_m})$ $\equiv 0$ mod $p$ or $\mathbf{F}(\gamma_{e_i}, \gamma_{e_{i+1}}, \ldots, \gamma_{e_m})$ $\equiv 0$ mod $p$ for all $\gamma_{v_1} \in \mathbb{Z}_p, \gamma_{v_2} \in \mathbb{Z}_p, \ldots, \gamma_{v_n} \in \mathbb{Z}_p, \gamma_{e_1} \in \mathbb{Z}_p, \gamma_{e_2} \in \mathbb{Z}_p, \ldots, \gamma_{e_m} \in \mathbb{Z}_p$. And, $\mathbf{F'}(v_1, v_2, \ldots, v_n, e_1,$ $e_2, \ldots, e_m)$ (or  $\mathbf{F'}(e_i, e_{i+1}, \ldots, e_m)$) is a zero polynomial (is $\equiv 0$ mod $p$) if there is no non-zero coefficient monomial available after applying the Fermat's theorem if exponent of variables is $\geq p$, that is, either $v_j^{p} \equiv v_j$ mod $p$ ($1 \leq j \leq n$) or $e_j^{p} \equiv e_j$ mod $p$ ($1 \leq j \leq m$) in $\mathbf{F}(v_1, v_2, \ldots, v_n,$ $e_1, e_2, \ldots, e_m)$ (or $\mathbf{F}(e_i, e_{i+1}, \ldots , e_m)$), in other words $\mathbf{F'}(v_1, v_2, \ldots, v_n, e_1,$ $e_2, \ldots, e_m)$ (or $\mathbf{F'}(e_i, e_{i+1},$ $\ldots , e_m)$) is $\not\equiv 0$ mod $p$ if there is a monomial whose coefficient is $\not\equiv 0$ mod $p$. For the sake of notational simplicity, let $\boldsymbol{e} =(e_1, e_2, \ldots , e_m)$.\\

\section{Algebraic settings and main results}
Given a graph $G$ with a vertex set $V(G) = \{v_1, v_2, \ldots, v_n\}$ and an edge set $E(G) = \{e_1, e_2, \ldots, e_m\}$, we define the polynomial $\mathbf{T}(v_1, v_2, \ldots, v_n, e_1, e_2, \ldots, e_m)$ over the finite field $\mathbb{Z}_p$ to find the \textit{total coloring} of the given graph.\\
\begin{multline*}
\mathbf{T}(v_1, v_2, \ldots, v_n, e_1, e_2, \ldots, e_m) = \prod_{i=1}^{n}\Bigg( \prod_{\substack{\textnormal{if}\ N_i(v_i) \neq \emptyset \\ v_j \in N_i(v_i)}}(v_i - v_j)\prod_{e_j \in N_e(v_i)}(v_i - e_j)\prod_{l = \Delta + 2}^{p}(v_i-l)\Bigg)\\ \prod_{i=1}^m\Bigg(\ \prod_{\substack{\textnormal{if}\ N_i(e_i) \neq \emptyset \\ e_j \in N_i(e_i)}}(e_i - e_j)\prod_{l \in \mathbb{Z}_p\setminus\mathcal{K}}(e_i - l) \Bigg).
\end{multline*}
Here finding a tuple $(\alpha_{v_1}, \alpha_{v_2}, \ldots, \alpha_{v_n}, \alpha_{e_1}, \alpha_{e_2}, \ldots, \alpha_{e_m})$, where $\alpha_{e_i}$ or $\alpha_{v_i} \in \mathcal{K}$ such that 
$\mathbf{T'}(\alpha_{v_1}, \alpha_{v_2}, \ldots, \alpha_{v_n}, \alpha_{e_1}, \alpha_{e_2}, \ldots, \alpha_{e_m}) \not\equiv 0$ mod $p$, gives us the \textit{total coloring} of the given graph as $\mathbf{T'}(\alpha_{v_1}, \alpha_{v_2}, \ldots, \alpha_{v_n}, \alpha_{e_1}, \alpha_{e_2}, \ldots, \alpha_{e_m}) \equiv \mathbf{T}(\alpha_{v_1}, \alpha_{v_2}, \ldots, \alpha_{v_n}, \alpha_{e_1}, \alpha_{e_2}, \ldots, \alpha_{e_m}) \not\equiv 0$ mod $p$. Thus far we have developed the algebraic setting $\mathbf{T}(v_1, v_2, \ldots, v_n, e_1, e_2, \ldots, e_m)$ and have no other clue on how to prove that $\mathbf{T'}(v_1, v_2, \ldots, v_n, e_1, e_2, \ldots, e_m) \not\equiv 0$ mod $p$. So we define the polynomials $\mathbf{P}(v_1, v_2, \ldots, v_n, e_1, e_2, \ldots, e_m)$ and $\mathbf{E}_m(e_1, e_2, \ldots, e_m)$ over the finite field $\mathbb{Z}_p$. These polynomials are the respective algebraic settings to find the \textit{vertex coloring} and the \textit{edge coloring} of the given graph which together define the \textit{total coloring}. Also these settings throw light on the motivation to construct the poloynomials and how colors are restricted to the set $\mathcal{K}$.

\begin{REMARK}
	The most important fact that we have to note of it while proving  $\mathbf{T'}(\alpha_{v_1}, \alpha_{v_2}, \ldots, \alpha_{v_n}, \alpha_{e_1}, \alpha_{e_2}, \ldots, \alpha_{e_m}) \equiv \mathbf{T}(\alpha_{v_1}, \alpha_{v_2}, \ldots, \alpha_{v_n}, \alpha_{e_1}, \alpha_{e_2}, \ldots, \alpha_{e_m}) \not\equiv 0$ mod $p$ is that not to assume that several multivariable polynomials are not identically zero means they all do not vanish at the same point.
\end{REMARK}

We define $\mathbf{P}(v_1, v_2, \ldots, v_n, e_1, e_2, \ldots, e_m)$ as follows,
\begin{equation*}
\label{e1}
\mathbf{P}(v_1, v_2, \ldots, v_n, e_1, e_2, \ldots, e_m) = \prod_{i=1}^{n}\Bigg( \prod_{\substack{\textnormal{if}\ N_i(v_i) \neq \emptyset \\ v_j \in N_i(v_i)}}(v_i - v_j)\prod_{e_j \in N_e(v_i)}(v_i - e_j)\prod_{l = \Delta + 2}^{p}(v_i-l)\Bigg).
\end{equation*}
If we say, $\mathcal{C}^{\mathbf{P'}}(\boldsymbol{e})$ is the coefficient of $\prod_{i=1}^nv_i^{l_i}$(for some $l_i \geq 0$, $1\leq i \leq n$) in 
$\mathbf{P'}(v_1, v_2, \ldots, v_n,$ $e_1, e_2, \ldots, e_m)$ then  $\mathcal{C}^{\mathbf{P'}}(\boldsymbol{e})$ is a polynomial of $e_1, e_2, \ldots, e_m$, not containing $v_i$'s ($1\leq i \leq n$).\\

Now, we make sure that $\mathbf{P'}(v_1, v_2, \ldots, v_n, e_1, e_2, \ldots, e_m) \not\equiv 0$ mod $p$, by proving the following theorem. Furthermore, proving the below stated Lemma \ref{l1} will guarantee that the exponent of each $e_i$ ($1 \leq i \leq m$) in 
$\mathcal{C}^{{\mathbf{P'}}}(\boldsymbol{e})$ is always less than or equal 
to 2.

\begin{THEOREM}
	\label{t1}
	There exists a coefficient $\mathcal{C}^{\mathbf{P'}}(\boldsymbol{e})$ $(\not\equiv 0\ mod\ p)$ of \ $\prod_{j=1}^{n}v_{j}^{l_j}$ $(for\ some \ l_1\geq0, l_2\geq0, \ldots, l_n\geq0)$ in $\mathbf{P'}(v_1, v_2, \ldots, v_n, e_1, e_2, \ldots, e_m)$.
\end{THEOREM}

\begin{LEMMA}
	\label{l1}
	The exponent of $e_k$'s $(1 \leq k \leq m)$ in  $\mathcal{C}^{\mathbf{P'}}(\boldsymbol{e})$ is always less than or equal 
	to 2.
\end{LEMMA}
Now, to define $\mathbf{E}_m(\boldsymbol{e})$, we first define $\mathbf{E}^i(e_i, e_{i+1}, \ldots, e_m)$ such that
\begin{equation*}
\label{e2}
\mathbf{E}^i(e_i, e_{i+1}, \ldots, e_m) = \prod_{\substack{\textnormal{if}\ N_i(e_i) \neq \emptyset \\ e_j \in N_i(e_i)}}(e_i - e_j)\prod_{l \in \mathbb{Z}_p\setminus\mathcal{K}}(e_i - l)
\end{equation*}
which leads to the polynomial
\begin{equation*}
\mathbf{E}_m(\boldsymbol{e}) = \prod_{i=1}^{m} \mathbf{E}^i(e_i, e_{i+1}, \ldots, e_m). 
\end{equation*}
Here, we can observe that finding a $m$-tuple $(\alpha_{e_1}, \alpha_{e_2}, \ldots, \alpha_{e_m})$, $\alpha_{e_i} \in \mathcal{K}$ $(1\leq i\leq m)$, such that $\mathbf{E}_m(\alpha_{e_1}, \alpha_{e_2}, \ldots, \alpha_{e_m})\not\equiv 0$ mod $p$,  
 is nothing but obtaining an \textit{edge coloring} of the given graph. Further, Vizing's theorem guarantee the existence of $m$-tuple $(\alpha_{e_1}, \alpha_{e_2}, \ldots, \alpha_{e_m})$, $\alpha_{e_i} \in \mathcal{K}$ $(1\leq i\leq m)$, such that $\mathbf{E}_m(\alpha_{e_1}, \alpha_{e_2}, \ldots, \alpha_{e_m})\not\equiv 0$ mod $p$, which is very essential. 

\begin{REMARK} 
 We remark that its because of the Vizing's theorem, we can say there exists $m$-tuple $(\alpha_{e_1}, \alpha_{e_2}, \ldots, \alpha_{e_m})$, $\alpha_{e_i} \in \mathcal{K}$ $(1\leq i\leq m)$, such that $\mathbf{E}_m(\alpha_{e_1}, \alpha_{e_2}, \ldots, \alpha_{e_m})\not\equiv 0$ mod $p$. In other words we  can say that, polynomials $\mathbf{Q}_i(\boldsymbol{e})$ $(1\leq i\leq m)$ or Claim \ref{cl3} (both defined/stated later) may face a serious question of existential crisis without Vizing's theorem guaranteeing the existence of $m$-tuple $(\alpha_{e_1}, \alpha_{e_2}, \ldots, \alpha_{e_m})$, $\alpha_{e_i} \in \mathcal{K}$ $(1\leq i\leq m)$, such that $\mathbf{E}_m(\alpha_{e_1}, \alpha_{e_2}, \ldots, \alpha_{e_m})\not\equiv 0$ mod $p$. \\
\end{REMARK}
Now, to prove the conjecture from the above algebraic settings of $\mathbf{P}(v_1, v_2, \ldots, v_n,$
 $e_1, e_2, \ldots, e_m)$ and $\mathbf{E}_m(\boldsymbol{e})$, our goal is to find a $m$-tuple $(\beta_{e_1}, \beta_{e_2}, \ldots, \beta_{e_m})$, $\beta_{e_i} \in \mathcal{K}$ $(1\leq i\leq m)$, such that $\mathbf{E}_m(\beta_{e_1}, \beta_{e_2}, \ldots, \beta_{e_m})\not\equiv 0$ mod $p$ and $\mathcal{C}^{\mathbf{P'}}(\beta_{e_1}, \beta_{e_2}, \ldots, \beta_{e_m}) \not\equiv 0$ mod $p$ as well. 
 The mapping $f(e_i) = \beta_{e_i}$ $(1\leq i \leq m)$ defines the edge coloring of $G$ using $\Delta+2$ colors. Since $\mathcal{C}^{\mathbf{P'}}(\beta_{e_1}, \beta_{e_2}, \ldots, \beta_{e_m})$ $(\not\equiv 0$ mod $p)$ is the coefficient of $\prod_{i=1}^nv_i^{l_i}$(for some $l_i \geq 0$, $1\leq i \leq n$) in 
 $\mathbf{P'}(v_1, v_2, \ldots, v_n,$ $e_1, e_2, \ldots, e_m)$, $\mathbf{P'}(v_1, v_2, \ldots, v_n, \beta_{e_1}, \beta_{e_2}, \ldots, \beta_{e_m}) \not\equiv 0$ mod $p$. Then we find a $n$-tuple $(\beta_{v_1}, \beta_{v_2}, \ldots, \beta_{v_n})$, $\beta_{v_i} \in \{1, 2, \ldots, \Delta +1\}$ $(1\leq i\leq n)$, such that $\mathbf{P'}(\beta_{v_1}, \beta_{v_2}, \ldots, \beta_{v_n}, \beta_{e_1}, \beta_{e_2}, \ldots, \beta_{e_m})\not\equiv 0$ mod $p$. This implies $\mathbf{T'}(\beta_{v_1}, \beta_{v_2}, \ldots, \beta_{v_n}, \beta_{e_1}, \beta_{e_2}, \ldots, \beta_{e_m})\not\equiv 0$ mod $p$. Thus we get the desired \textit{total coloring} of the given graph.\\

\begin{REMARK}
	\label{pl2}
	Before going further we remark that the polynomial $\mathbf{Z}_i(\boldsymbol{e})$ defined below, will help us to realise why the  \textit{total chromatic number} $\chi''(G)$ of the given graph is $\geq \Delta + 2$.\\
	
	To define $\mathbf{Z}_i(\boldsymbol{e})$, for $1 \leq i \leq m$, given an edge $e_i = \{v_s, v_t\}$ \textnormal{(}without loss of generality we assume $|N(v_s)| \geq |N(v_t)|$\textnormal{)}, let $\mathbf{S}_i = \{e_{l_1}, e_{l_2}, \ldots, e_{l_r} :e_{l_j} \in N_e(v_s), 1\leq l_1, l_2, \ldots, l_r \leq m\}$ and we have $|\mathbf{S}_i| \leq \Delta$. And $\mathbf{Z}_i(\boldsymbol{e})$ is defined as follows
	\begin{equation*}
	\mathbf{Z}_i(\boldsymbol{e}) = \mathcal{C}^{\mathbf{P'}}(\boldsymbol{e}) \prod_{1 \leq j < k \leq r}(e_{l_j} - e_{l_k})
	\prod_{e_{l_j} \in \mathbf{S}_i}(\prod_{l = \Delta + 3}^p(e_{l_j} - l)).
	\end{equation*}
	
	Now, proving the following statement will assert our remark:\\
	
	For $1 \leq i \leq m$, the polynomial $\mathbf{Z}_i(\boldsymbol{e})  \not\equiv 0$ mod $p$.
	
\end{REMARK}

Further to find a $m$-tuple $(\beta_{e_1}, \beta_{e_2}, \ldots, \beta_{e_m})$ $(\beta_{e_i} \in \mathcal{K})$ such that 
$\mathbf{E}_m(\beta_{e_1}, \beta_{e_2}, \ldots, \beta_{e_m})\not\equiv 0$ mod $p$ and $\mathcal{C}^{\mathbf{P'}}(\beta_{e_1}, \beta_{e_2}, \ldots, \beta_{e_m}) \not\equiv 0$ mod $p$ as well, we define the polynomial  
$\mathbf{Q}_i(\boldsymbol{e})$ as follows.\\

Let $\mathbf{Q}_1(\boldsymbol{e}) = \mathcal{C}^{\mathbf{P'}}(\boldsymbol{e})\mathbf{E}^1(e_1, e_{2}, \ldots, e_m).$\\

And, for $2 \leq i \leq m$, let \\
\begin{equation*}
\label{eq}
\mathbf{Q}_i(\boldsymbol{e}) = \mathbf{Q}_{i-1}(\boldsymbol{e})\mathbf{E}^i(e_i, e_{i+1}, \ldots, e_m).
\end{equation*}

From the above definition of  $\mathbf{Q}_i(\boldsymbol{e})$, it can be observed that finding a $m$-tuple $(\beta_{e_1}, \beta_{e_2}, \ldots, \beta_{e_m})$ $(\beta_{e_i} \in \mathcal{K})$ such that $\mathbf{Q}_m(\beta_{e_1}, \beta_{e_2}, \ldots, \beta_{e_m}) \not\equiv 0$ mod $p$, that is, $\mathcal{C}^{\mathbf{P'}}(\beta_{e_1}, \beta_{e_2}, \ldots, \beta_{e_m})\mathbf{E}_m(\beta_{e_1}, \beta_{e_2}, \ldots, \beta_{e_m})\not\equiv 0$ mod $p$,
is nothing but obtaining an \textit{edge coloring} of the given graph using $\Delta+2$ colors, while  also establishing that 
$\mathcal{C}^{\mathbf{P'}}(\beta_{e_1}, \beta_{e_2}, \ldots, \beta_{e_m}) \not\equiv 0$ mod $p$. Further, to prove $\mathbf{Q}_m(\beta_{e_1}, \beta_{e_2}, \ldots, \beta_{e_m}) \not\equiv 0$ mod $p$, we have to prove the following claims,

\begin{HYPOTHESIS}
	\label{h1}
	Suppose 
	there exists $(\beta_{e_1}, \beta_{e_2}, \ldots, \beta_{e_{i-1}}, e_i, \ldots, e_m)$, $\beta_{e_j} \in \mathcal{K} = \{1, 2, \dots, \Delta+1, \alpha = \alpha_{i-1}\} (1\leq j \leq i-1),\ where\ \alpha_{i-1} \in \mathbb{Z}_p\setminus\{0, 1, 2, \ldots, \Delta+1\}$,
	such that $\mathbf{Q}_{i-1}(\beta_{e_1}, \beta_{e_2}, \ldots, \beta_{e_{i-1}}, e_i, \ldots, e_m) \not\equiv 0$ mod $p$, that is, $\mathbf{Q'}_{i-1}(\boldsymbol{e}) \not\equiv 0$ mod $p$, is nothing but, 
	\begin{multline*}
	\mathcal{C}^{{\mathbf{P'}}}(\boldsymbol{e})\prod_{j=1}^{i-1}\Big(\prod_{\substack{\textnormal{if}\ N_j(e_j) \neq \emptyset \\ e_l \in N_j(e_j)}}({e_j} - e_l)\prod_{k \in \mathbb{Z}_p\setminus\{1, 2, \ldots, \Delta+1, \alpha = \alpha_{i-1}\}}({e_j} - k)\Big)  \not\equiv 0\ \textnormal{mod}\ p,\\
	\end{multline*}
	but there does not exists $(\beta'_{e_1}, \beta'_{e_2}, \ldots, \beta'_{e_{i-1}}, \beta'_{e_i}, e_{i+1},  \ldots, e_m)$, $\beta'_{e_j} \in \mathcal{K} = \{1, 2, \dots, \Delta+1, \alpha = \alpha_{i-1}\} (1\leq j \leq i),\ where\ \alpha_{i-1} \in \mathbb{Z}_p\setminus\{0, 1, 2, \ldots, \Delta+1\}$, such that $\mathbf{Q}_i(\beta'_{e_1}, \beta'_{e_2}, \ldots, \beta'_{e_{i-1}}, \beta'_{e_i}, e_{i+1}, \ldots, e_m) \not\equiv 0$ mod $p$.
	In other words, $\mathbf{Q'}_i(\boldsymbol{e}) \equiv 0$ mod $p$, is nothing but\\
	\begin{multline*}
	\mathcal{C}^{{\mathbf{P'}}}(\boldsymbol{e})\prod_{j=1}^{i}\Big(\prod_{\substack{\textnormal{if}\ N_j(e_j) \neq \emptyset \\ e_l \in N_j(e_j)}}({e_j} - e_l)\prod_{k \in \mathbb{Z}_p\setminus\{1, 2, \ldots, \Delta+1, \alpha = \alpha_{i-1}\}}({e_j} - k)\Big)  \equiv 0\ \textnormal{mod}\ p.\\
	\end{multline*}	
\end{HYPOTHESIS} 
Then we are finding a new value to $\alpha$, say $\alpha = \beta_i$, such that the following claim is true and cardinality of $\mathcal{K}$ remains same as well,
\begin{CLAIM}
	\label{cl3}
	There exists $(\beta'_{e_1}, \beta'_{e_2}, \ldots, \beta'_{e_{i-1}}, \beta'_{e_i}, e_{i+1},  \ldots, e_m)$, $\beta'_{e_j}\in \mathcal{K} = \{1, 2, \dots, \Delta+1, \alpha = \beta_i\}) (1\leq j \leq i), \ where\ \beta_i \in \mathbb{Z}_p\setminus\{0, 1, 2, \ldots, \Delta+1, \alpha_{i-1}\}$, such that $\mathbf{Q}_i(\beta'_{e_1}, \beta'_{e_2}, \ldots, \beta'_{e_{i-1}}, \beta'_{e_i}, e_{i+1}, \ldots, e_m) \not\equiv 0$ mod $p$.
	In other words, $\mathbf{Q'}_i(\boldsymbol{e}) \not\equiv 0$ mod $p$, is nothing but\\
	\begin{multline*}
	\mathcal{C}^{{\mathbf{P'}}}(\boldsymbol{e})\prod_{j=1}^{i}\Big(\prod_{\substack{\textnormal{if}\ N_j(e_j) \neq \emptyset \\ e_l \in N_j(e_j)}}({e_j} - e_l)\prod_{k \in \mathbb{Z}_p\setminus\{1, 2, \ldots, \Delta+1, \alpha = \beta_i\}}({e_j} - k)\Big)  \not\equiv 0\ \textnormal{mod}\ p.\\
	\end{multline*}
\end{CLAIM}

\begin{REMARK}
	$$\textnormal{Let}\ F(x) = x\prod_{\gamma \in \mathbb{Z}_p\setminus\{0\}}(x-\gamma) \equiv  x(x^{p-1}-1) \equiv x^p - x.\ \textnormal{Then}\ F'(x) \equiv x-x \equiv 0.$$
\end{REMARK}
\begin{REMARK}
	In this remark we will give a brief view and explaination of the strategy adopted to prove Claim \ref{cl3}. Now we will focus in understanding the scenario when Hyphothesis \ref{h1} is true. From the Hypothesis \ref{h1} we have that $\mathbf{Q'}_{i-1}(\boldsymbol{e}) \not\equiv 0$ mod $p$, is nothing but, 
\begin{multline*}
\mathcal{C}^{{\mathbf{P'}}}(\boldsymbol{e})\prod_{j=1}^{i-1}\Big(\prod_{\substack{\textnormal{if}\ N_j(e_j) \neq \emptyset \\ e_l \in N_j(e_j)}}({e_j} - e_l)\prod_{k \in \mathbb{Z}_p\setminus\{1, 2, \ldots, \Delta+1, \alpha = \alpha_{i-1}\}}({e_j} - k)\Big)  \not\equiv 0\ \textnormal{mod}\ p.\\
\end{multline*}
And $\mathbf{Q'}_{i-1}(\boldsymbol{e})$ can be written as follows (after applying Feramt's theroem),
\begin{equation}
\label{req1} 
\mathbf{Q'}_{i-1}(\boldsymbol{e}) \equiv \sum_{\prod_{\substack{r=1\\r\neq i}}^ml_r}\mathcal{C}_{\prod_{\substack{r=1\\r\neq i}}^ml_r}(e_i)\prod_{\substack{r = 1\\r \neq i}}^{m}{e_r}^{l_r},
\end{equation}
where $\mathcal{C}_{\prod_{\substack{r=1\\r\neq i}}^ml_r}(e_i)$ (it is either an univariate polynomial in $e_i$ or constant) is the coefficient of $\prod_{\substack{r = 1\\r \neq i}}^{m}{e_r}^{l_r}$, exponent of each $e_r$ $(\substack{1 \leq r \leq m \\r\neq i})$ and $e_i$ is $\leq p-1$.\\
By the definition of $\mathbf{Q}_{i}(\boldsymbol{e})$, we have, 
\begin{equation*}
\mathbf{Q}_{i}(\boldsymbol{e}) = \mathbf{Q}_{i-1}(\boldsymbol{e})\mathbf{E}^{i}(e_i, e_{i+1}, \ldots, e_m)
\end{equation*}
$\mathbf{Q'}_{i}(\boldsymbol{e})$ can also be written as follows using relation (\ref{req1}),
\begin{equation*}
\mathbf{Q'}_{i}(\boldsymbol{e}) \equiv \Big(\sum_{\prod_{\substack{r=1\\r\neq i}}^ml_r}\mathcal{C}_{\prod_{\substack{r=1\\r\neq i}}^ml_r}(e_i)\prod_{\substack{r = 1\\r \neq i}}^{m}{e_r}^{l_r}\Big)\prod_{\substack{\textnormal{if}\ N_i(e_i) \neq \emptyset \\ e_j \in N_i(e_i)}}(e_i - e_j)\prod_{l \in \mathbb{Z}_p\setminus\mathcal{K} = \{1, 2, \dots, \Delta+1, \alpha = \alpha_{i-1}\}}(e_i - l),
\end{equation*}
\begin{equation*}
\mathbf{Q'}_{i}(\boldsymbol{e}) \equiv \Big(\sum_{\prod_{\substack{r=1\\r\neq i}}^ml_r}\big(\mathcal{C}_{\prod_{\substack{r=1\\r\neq i}}^ml_r}(e_i)\prod_{l \in \mathbb{Z}_p\setminus\mathcal{K} = \{1, 2, \dots, \Delta+1, \alpha = \alpha_{i-1}\}}(e_i - l)\big)\prod_{\substack{r = 1\\r \neq i}}^{m}{e_r}^{l_r}\Big)\prod_{\substack{\textnormal{if}\ N_i(e_i) \neq \emptyset \\ e_j \in N_i(e_i)}}(e_i - e_j),
\end{equation*}
where $\mathcal{C}_{\prod_{\substack{r=1\\r\neq i}}^ml_r}(e_i)$ is the coefficient of $\prod_{\substack{r = 1\\r \neq i}}^{m}{e_r}^{l_r}$, exponent of each $e_r$ $(\substack{1 \leq r \leq m \\r\neq i})$ is $\leq p-1$, and as Fermat's theorem has not been applied to $e_i$ yet, exponent of $e_i$ is $\leq p+\Delta-2$.\\
Further, $\mathbf{Q'}_{i}(\boldsymbol{e})$ can also be written as follows,
\begin{equation*}
\mathbf{Q'}_{i}(\boldsymbol{e}) \equiv \Big(\sum_{\prod_{\substack{r=1\\r\neq i}}^ml_r}\mathcal{C}_{\prod_{\substack{r=1\\r\neq i}}^ml_r}^1(e_i)\prod_{\substack{r = 1\\r \neq i}}^{m}{e_r}^{l_r}\Big)\prod_{\substack{\textnormal{if}\ N_i(e_i) \neq \emptyset \\ e_j \in N_i(e_i)}}(e_i - e_j),
\end{equation*}
where $\mathcal{C}_{\prod_{\substack{r=1\\r\neq i}}^ml_r}^1(e_i) = \mathcal{C}_{\prod_{\substack{r=1\\r\neq i}}^ml_r}(e_i)\prod_{l \in \mathbb{Z}_p\setminus\mathcal{K} = \{1, 2, \dots, \Delta+1, \alpha = \alpha_{i-1}\}}(e_i - l)$(it is an univariate polynomial in $e_i$).\\
We know that $\mathcal{C'}_{\prod_{\substack{r=1\\r\neq i}}^ml_r}^1(e_i)$ is a polynomial obtained after applying Fermat's theorem to the univariate polynomial $\mathcal{C}_{\prod_{\substack{r=1\\r\neq i}}^ml_r}^1(e_i)$, using this fact, $\mathbf{Q'}_{i}(\boldsymbol{e})$ can also be written as follows,
\begin{equation}
\label{eqt1}
\mathbf{Q'}_{i}(\boldsymbol{e}) \equiv \Big(\sum_{\prod_{\substack{r=1\\r\neq i}}^ml_r}\mathcal{C'}_{\prod_{\substack{r=1\\r\neq i}}^ml_r}^1(e_i)\prod_{\substack{r = 1\\r \neq i}}^{m}{e_r}^{l_r}\Big)\prod_{\substack{\textnormal{if}\ N_i(e_i) \neq \emptyset \\ e_j \in N_i(e_i)}}(e_i - e_j).
\end{equation}
Here, we must note that the only possible reason for $\mathbf{Q'}_{i}(\boldsymbol{e}) \equiv 0 $ mod $p$ (that is Hypothesis \ref{h1} is true) from the relation (\ref{eqt1}) is that each $\mathcal{C'}_{\prod_{\substack{r=1\\r\neq i}}^ml_r}^1(e_i) \equiv 0$ mod $p$, that is,
\begin{equation*}
\mathbf{Q'}_{i}(\boldsymbol{e}) \equiv \Big(\sum_{\prod_{\substack{r=1\\r\neq i}}^ml_r}\underbrace{\mathcal{C'}_{\prod_{\substack{r=1\\r\neq i}}^ml_r}^1(e_i)}_{\equiv\ 0\ \text{mod}\ p}\prod_{\substack{r = 1\\r \neq i}}^{m}{e_r}^{l_r}\Big)\prod_{\substack{\textnormal{if}\ N_i(e_i) \neq \emptyset \\ e_j \in N_i(e_i)}}(e_i - e_j) \equiv 0\ \text{mod}\ p.
\end{equation*}
Otherwise, suppose there exists a monomial $\prod_{\substack{r = 1\\r \neq i}}^{m}{e_r}^{l'_r}$ such that its coefficient $\mathcal{C'}_{\prod_{\substack{r=1\\r\neq i}}^ml'_r}^1(e_i) \not\equiv 0$ mod $p$ in the relation (\ref{eqt1}), that is,
\begin{equation*}
	\mathbf{Q'}_{i}(\boldsymbol{e}) \equiv \underbrace{\Big(\underbrace{\mathcal{C'}_{\prod_{\substack{r=1\\r\neq i}}^ml'_r}^1(e_i)}_{\not\equiv\ 0\ \text{mod}\ p}\prod_{\substack{r = 1\\r \neq i}}^{m}{e_r}^{l'_r} + \sum_{\prod_{\substack{r=1\\r\neq i}}^ml_r}\mathcal{C'}_{\prod_{\substack{r=1\\r\neq i}}^ml_r}^1(e_i)\prod_{\substack{r = 1\\r \neq i}}^{m}{e_r}^{l_r}\Big)}_{\text{we denote the underbrace polynomial by}\ \mathbf{W}_{i}(\boldsymbol{e})}\prod_{\substack{\textnormal{if}\ N_i(e_i) \neq \emptyset \\ e_j \in N_i(e_i)}}(e_i - e_j),
\end{equation*}
that is, 
\begin{equation*}
\mathbf{Q'}_{i}(\boldsymbol{e}) \equiv \mathbf{W}_{i}(\boldsymbol{e})\prod_{\substack{\textnormal{if}\ N_i(e_i) \neq \emptyset \\ e_j \in N_i(e_i)}}(e_i - e_j),
\end{equation*}
so we have that $\mathbf{W}_{i}(\boldsymbol{e}) \not\equiv 0$ mod $p$, to each variable $e_1$, $e_2$, \ldots, $e_m$ in $\mathbf{W}_{i}(\boldsymbol{e})$ we associate the sets $\mathcal{A}_1 = \mathbb{Z}_p,  \mathcal{A}_{2} = \mathbb{Z}_p, \ldots,$ $\mathcal{A}_{m} = \mathbb{Z}_p$ respectively. Then there exists $\gamma_{e_1} \in \mathcal{A}_1, \gamma_{e_2} \in \mathcal{A}_{2}, \ldots,$ $\gamma_{e_m} \in \mathcal{A}_{m}$ such that

\begin{equation*}
\mathbf{W}(\gamma_{e_1}, \gamma_{e_2}, \ldots, \gamma_{e_i}, \gamma_{e_{i+1}}, \ldots, \gamma_{e_m}) \not\equiv 0\ \textnormal{mod}\ p,
\end{equation*}
also following holds as well,
\begin{equation*}
\mathbf{W}(\gamma_{e_1}, \gamma_{e_2}, \ldots, \gamma_{e_i}, e_{i+1}, \ldots, {e_m}) \not\equiv 0\ \textnormal{mod}\ p.
\end{equation*}

So, we have 
\begin{equation*}
\mathbf{Q'}_{i}(\gamma_{e_1}, \gamma_{e_2}, \ldots, \gamma_{e_i}, e_{i+1}, \ldots, e_m) \equiv \mathbf{W}(\gamma_{e_1}, \gamma_{e_2}, \ldots, \gamma_{e_i}, e_{i+1}, \ldots, e_m)\prod_{\substack{\textnormal{if}\ N_i(e_i) \neq \emptyset \\ e_j \in N_i(e_i)}}(\gamma_{e_i} - e_j).
\end{equation*}
Let \ $\prod_{j = i+1}^ me_j^{l_j}$ ($0 \leq l_{j} \leq 2\Delta$) be a monomial of maximum degree in $\mathbf{W}(\gamma_{e_1}, \gamma_{e_2}, \ldots, \gamma_{e_i}, e_{i+1}, \ldots, e_m)$, and we have that   
$\prod_{e_j \in N_i(e_i)}e_j$  is the only monomial (unique) of maximum degree in $$\prod_{\substack{\textnormal{if}\ N_i(e_i) \neq \emptyset \\ e_j \in N_i(e_i)}}(\gamma_{e_i} - e_j).$$\\

The product of the above two monomials give a unique monomial of maximum degree whose coefficient is non-zero ($\not\equiv 0$ mod $p$) in $\mathbf{Q'}_{i}(\gamma_{e_1}, \gamma_{e_2}, \ldots, \gamma_{e_i}, e_{i+1}, \ldots, e_m)$. Therefore, $\mathbf{Q'}_{i}(\gamma_{e_1}, \gamma_{e_2}, \ldots, \gamma_{e_i}, e_{i+1}, \ldots, e_m)$ is not a zero polynomial, that is,
\begin{equation*}
\mathbf{Q'}_{i}(\gamma_{e_1}, \gamma_{e_2}, \ldots, \gamma_{e_i}, e_{i+1}, \ldots, e_m)\not\equiv 0\ \textnormal{mod}\ p,
\end{equation*}
so we can say that,
\begin{equation*}
\mathbf{Q'}_{i}(e_1, e_2, \ldots, e_i, e_{i+1}, \ldots, e_m)\not\equiv 0\ \textnormal{mod}\ p,
\end{equation*}
which implies that Hypothesis \ref{h1} is false.\\

So, if the Hypothesis \ref{h1} is true, we find a new value to $\alpha$, say $\alpha = \beta_i$. In the following paragraphs we explain the strategy adopted in detail to find $\beta_i$, especially, the paragraph after defining $\mathbf{H}_{e_j}(\boldsymbol{e})$ gives us brief view that how this makes sure that $\mathbf{Q'}_{i}(\boldsymbol{e}) \not\equiv 0 $ mod $p$.\\ 
\end{REMARK}

Now, our goal is to find a new value to $\alpha$, say $\alpha = \beta_i$. We will find a new value $\alpha = \beta_i$ using the polynomials $\mathbf{J}_{\prod_{\substack{r=1\\r\neq j}}^ml_r}(e_j)$ defined later. To define $\mathbf{J}_{\prod_{\substack{r=1\\r\neq j}}^ml_r}(e_j)$ we need to define the polynomials $\mathbf{G}(\boldsymbol{e})$ and $\mathbf{H}_{e_j}(\boldsymbol{e})$ as follows, we use the fact that $\mathbf{Q'}_{i-1}(\boldsymbol{e}) \not\equiv 0$ mod $p$, that is, 
\begin{multline}
\label{m7}
\mathcal{C}^{{\mathbf{P'}}}(\boldsymbol{e})\prod_{j=1}^{i-1}\Bigg(\prod_{\substack{\textnormal{if}\ N_j(e_j) \neq \emptyset \\ e_l \in N_j(e_j)}}({e_j} - e_l)\prod_{k \in \mathbb{Z}_p\setminus\{1, 2, \ldots, \Delta+1, \alpha = \alpha_{i-1}\}}({e_j} - k)\Bigg)  \not\equiv 0\ \textnormal{mod}\ p.\\
\end{multline}
\begin{equation*}
\mathbf{G}(\boldsymbol{e}) = \mathbf{Q}_{i-1}(\boldsymbol{e})\prod_{j=1}^{k}(e_{l_j}-\alpha_{i-1}),
\end{equation*}
where $e_{l_1}, e_{l_2}, \ldots, e_{l_k} \in \{e_1, e_2, \ldots, e_{i-1}\}$ $(1 \leq l_1, l_2, \ldots, l_k \leq i-1)$ such that 
$\mathbf{G'}(\boldsymbol{e}) \not\equiv 0$ mod $p$ and for every $e_j \in \{e_1, e_2, \ldots, e_{i-1}\}\setminus\{e_{l_1}, e_{l_2}, \ldots, e_{l_k}\}$ the polynomial $\mathbf{H'}_{e_j}(\boldsymbol{e}) \equiv 0$ mod $p$, where, 
\begin{equation*}
\mathbf{H}_{e_j}(\boldsymbol{e}) = \mathbf{G'}(\boldsymbol{e})(e_j-\alpha_{i-1}).\\
\end{equation*}

Our novel approach involves finding a new value to $\alpha$, say $\alpha =\beta_i$ $(\notin \{0, 1, 2, \ldots, \Delta+1, \alpha_{i-1}\}$), such that for each $e_j \in \{e_1, e_2, \ldots, e_{i-1}\}\setminus\{e_{l_1}, e_{l_2}, \ldots, e_{l_k}\}$, $(e_j-\beta_i)$ is a  square-free factor in $\mathbf{J}_{\prod_{\substack{r=1\\r\neq j}}^ml_r}(e_j)$ (defined later) and $(e_j-\beta_i)$ is a  square-free factor in $\mathbf{G'}(\boldsymbol{e})$ as well (by proving the Claim \ref{cl2} stated later). And for each $e_j \in \{e_1, e_2, \ldots, e_{i-1}\}\setminus\{e_{l_1}, e_{l_2}, \ldots, e_{l_k}\}$, the square-free factor $(e_j-\beta_i)$ in $\mathbf{G'}(\boldsymbol{e})$ is replaced by $\big(e_j- \alpha_{i-1}\big)$ to obtain the desired result (by proving the Claim \ref{cl3}). For the sake of notational simplicity, let $M_1 = \{e_{l_1}, e_{l_2}, \ldots, e_{l_k}\}$ and $M_2 = \{e_1, e_2, \ldots, e_{i-1}\}\setminus\{e_{l_1}, e_{l_2}, \ldots, e_{l_k}\}$ $(1 \leq l_1, l_2, \ldots, l_k \leq i-1)$. To choose a new value $\beta_i$, we are defining polynomial $\mathbf{J}_{\prod_{\substack{r=1\\r\neq j}}^ml_r}(e_j)$ as follows.\\

By the definition of $\mathbf{G}(\boldsymbol{e})$ and $\mathbf{Q}_{i-1}(\boldsymbol{e})$, $\mathbf{G}(\boldsymbol{e})$ can be written as follows,
\small{ 
	\begin{multline}
	\label{m8}
	\mathbf{G}(\boldsymbol{e}) = \Bigg(\mathcal{C}^{{\mathbf{P'}}}(\boldsymbol{e})\prod_{j=1}^{i-1}\Big(\prod_{\substack{\textnormal{if}\ N_j(e_j) \neq \emptyset \\ e_l \in N_j(e_j)}}({e_j} - e_l)\prod_{k \in \mathbb{Z}_p\setminus\{1, 2, \ldots, \Delta+1, \alpha = \alpha_{i-1}\}}({e_j} - k)\Big)\Bigg)\Bigg(\prod_{j=1}^{k}(e_{l_j}-\alpha_{i-1})\Bigg).
	\end{multline}}

Given an $e_j \in M_2$, the polynomial $\mathbf{G'}(\boldsymbol{e})$ can be written as (Fermat's theorem is applied to each $e_r$ ($1 \leq r \leq m $)),
\begin{equation}
\label{eq11}
\mathbf{G'}(\boldsymbol{e}) \equiv \sum_{\prod_{\substack{r=1\\r\neq j}}^ml_r}\mathcal{C}_{\prod_{\substack{r=1\\r\neq j}}^ml_r}(e_j)\prod_{\substack{r = 1\\r \neq j}}^{m}{e_r}^{l_r},
\end{equation}
where $\mathcal{C}_{\prod_{\substack{r=1\\r\neq j}}^ml_r}(e_j)$ is the coefficient of $\prod_{\substack{r = 1\\r \neq j}}^{m}{e_r}^{l_r}$, exponent of each $e_r$ $(\substack{1 \leq r \leq m \\r\neq j})$ and $e_j$ is $\leq p-1$.\\
Also we can observe that, given an $e_j \in M_2$, from the relation (\ref{m8}) 
$\mathbf{G'}(\boldsymbol{e})$
can also be written as (excepting $e_j$, Fermat's theorem is applied to each $e_r$ ($\substack{1 \leq r \leq m \\r\neq j})$), 
\begin{equation}
\label{m3}
\mathbf{G'}(\boldsymbol{e}) \equiv  \sum_{\prod_{\substack{r=1\\r\neq j}}^ml_r}\Big(\mathcal{C}^{j}_{\prod_{\substack{r=1\\r\neq j}}^ml_r}(e_j)\prod_{l\in \mathbb{Z}_p\setminus\{ 1, 2, \ldots, \Delta+1, \alpha=\alpha_{i-1}\}}(e_j-l)\Big)\prod_{\substack{r = 1\\r \neq j}}^{m}{e_r}^{l_r},
\end{equation}
where $0 \leq l_r \leq p-1$ $(\textnormal{as Fermat's theorem is applied to each}\ e_r\ (\substack{1 \leq r \leq m \\r\neq j}), \textnormal{except}\ e_j)$, $\mathcal{C}^j_{\prod_{\substack{r=1\\r\neq j}}^ml_r}(e_j)$ is a univariate polynomial of $e_j$ and exponent of $e_j$ in the following products, $$\Big(\mathcal{C}^j_{\prod_{\substack{r=1\\r\neq j}}^ml_r}(e_j)\prod_{l\in \mathbb{Z}_p\setminus\{ 1, 2, \ldots, \Delta+1, \alpha=\alpha_{i-1}\}}(e_j-l)\Big)\ \textnormal{is}\ \leq (2\Delta+p-(\Delta+2)) = p+\Delta-2.$$\\

The most important fact we observe from relation (\ref{eq11}) and relation (\ref{m3}) is that, on applying Fermat's theorem to $e_j$, 
\begin{equation}
\label{eq12}
\mathcal{C}^j_{\prod_{\substack{r=1\\r\neq j}}^ml_r}(e_j)\prod_{l\in \mathbb{Z}_p\setminus\{ 1, 2, \ldots, \Delta+1, \alpha=\alpha_{i-1}\}}(e_j-l) \equiv \mathcal{C}_{\prod_{\substack{r=1\\r\neq j}}^ml_r}(e_j),
\end{equation}
as both are coefficient of $\prod_{\substack{r = 1\\r \neq j}}^{m}{e_r}^{l_r}$.\\ \\

Now, for each $e_j \in M_2$, we define polynomial $\mathbf{J}_{\prod_{\substack{r=1\\r\neq j}}^ml_r}(e_j)$ using a non-zero coefficient of some monomial $\prod_{\substack{r = 1\\r \neq j}}^{m}{e_r}^{l_r}$ in the congruence relation (\ref{m3}) as follows,\\ \\

for each $e_j \in M_2$,
\begin{equation}
\label{eq3}
\mathbf{J}_{\prod_{\substack{r=1\\r\neq j}}^ml_r}(e_j) = \mathcal{C}^j_{\prod_{\substack{r=1\\r\neq j}}^ml_r}(e_j)\prod_{l\in \mathbb{Z}_p\setminus\{ 1, 2, \ldots, \Delta+1, \alpha=\alpha_{i-1}\}}(e_j-l),
\end{equation}
where, $$\mathcal{C}^j_{\prod_{\substack{r=1\\r\neq j}}^ml_r}(e_j)\prod_{l\in \mathbb{Z}_p\setminus\{ 1, 2, \ldots, \Delta+1, \alpha=\alpha_{i-1}\}}(e_j-l)\ \textnormal{is}\ $$ is the coefficient 
of some monomial $\prod_{\substack{r=1\\r\neq j}}^me_r^{l_r}$ in the congruence relation (\ref{m3}).\\
And $\mathbf{J}_{\prod_{\substack{r=1\\r\neq i}}^ml_r}(e_i)$ is defined below,
\begin{equation}
\label{eq10}
\mathbf{J}_{\prod_{\substack{r=1\\r\neq i}}^ml_r}(e_i) = \mathcal{C}^i_{\prod_{\substack{r=1\\r\neq i}}^ml_r}(e_i),
\end{equation}
where, $\mathcal{C}^i_{\prod_{\substack{r=1\\r\neq i}}^ml_r}(e_i)$ is the coefficient 
of some monomial $\prod_{\substack{r=1\\r\neq i}}^me_r^{l_r}$ in the congruence relation (\ref{m3}).\\

Now, with the help of the polynomials $\{\mathbf{J}_{\prod_{\substack{r=1\\r\neq j}}^ml_r}(e_j):e_j \in M_2\} \cup \{\mathbf{J}_{\prod_{\substack{r=1\\r\neq i}}^ml_r}(e_i)\}$ (as defined in (\ref{eq3}) and (\ref{eq10})), we will find the new value to $\alpha$, that is, $\alpha = \beta_i$, in the following claim,

\begin{CLAIM}
	\label{cl2}
	There exists $\beta_{i}\in \mathbb{Z}_p\setminus\{0, 1, 2, \ldots, \Delta+1, \alpha_{i-1}\}$ such that for each $e_j \in M_2$, $(e_j-\beta_i)$ divides $\mathbf{J}_{\prod_{\substack{r=1\\r\neq j}}^ml_r}(e_j)$ but $(e_j-\beta_i)^2$ does not divide $\mathbf{J}_{\prod_{\substack{r=1\\r\neq j}}^ml_r}(e_j)$. Moreover, for each $e_j \in M_2$, $(e_j-\beta_i)$ divides
	$\mathbf{G'}(\boldsymbol{e})$ 
	but $(e_j-\beta_i)^2$ does not divide $\mathbf{G'}(\boldsymbol{e})$. And also $(e_i-\beta_i)$ does not divide $\mathbf{J}_{\prod_{\substack{r=1\\r\neq i}}^ml_r}(e_i)$ and $\mathbf{G'}(\boldsymbol{e})$ 
	as well.
\end{CLAIM}

To prove the Claim \ref{cl2} we have to prove the following Lemma \ref{l2}, and to prove the Claim \ref{cl3} we have to prove the Lemma \ref{l3}. 

\begin{LEMMA}
	\label{l2}
	Given an $e_j \in M_2$, $\mathbf{H}_{e_j}(\boldsymbol{e}) = \mathbf{G'}(\boldsymbol{e})(e_j-\alpha_{i-1})$. For every $e_j \in M_2$, the polynomial $\mathbf{H'}_{e_j}(\boldsymbol{e}) \equiv 0$ mod $p$. Then  
	\begin{equation*}
	\mathbf{G'}(\boldsymbol{e}) \equiv \prod_{e_j \in M_2}\Big(\prod_{l \in \mathbb{Z}_p\setminus\{\alpha_{i-1}\}}(e_j-l)\Big)\bigg(\sum_{\prod_{\substack{r=1\\r \notin K}}^ml_r}\mathcal{C}_{\prod_{\substack{r=1\\r \notin K}}^ml_r}(e_i)\prod_{\substack{r = 1\\r \notin K}}^{m}{e_r}^{l_r}\bigg),
	\end{equation*}
	where, $K = \{s: e_s \in M_2\} \cup \{i\}$ and $\mathcal{C}_{\prod_{\substack{r=1\\r \notin K}}^ml_r}(e_i)$ is a univariate polynomial in $e_i$.
\end{LEMMA}
\begin{LEMMA}
	\label{l3}
	\begin{equation*}
	\frac{\mathbf{G}(\boldsymbol{e})}{\prod_{e_j \in M_2}\Big(e_j-\beta_i\Big)} \equiv \frac{\mathbf{G'}(\boldsymbol{e})}{\prod_{e_j \in M_2}\Big(e_j-\beta_i\Big)},
	\end{equation*}
	that is, 
	\begin{multline*}
	\mathcal{C}^{\mathbf{P'}}(\boldsymbol{e}) \prod_{e_j \in M_1} \Big(\prod_{\substack{\textnormal{if}\ N_j(e_j) \neq \emptyset \\ e_l \in N_j(e_j)}}({e_j} - e_l)\prod_{k \in \mathbb{Z}_p\setminus\{1, 2, \ldots, \Delta+1\}}({e_j} - k)\Big) \\ \prod_{e_j \in M_2} \Big(\prod_{\substack{\textnormal{if}\ N_j(e_j) \neq \emptyset \\ e_l \in N_j(e_j)}}({e_j} - e_l)\prod_{k \in \mathbb{Z}_p\setminus\{1, 2, \ldots, \Delta+1, \alpha_{i-1}, \beta_i\}}({e_j} - k)\Big) \equiv \\ \prod_{e_j \in M_2}\Big(\prod_{l \in \mathbb{Z}_p\setminus\{\alpha_{i-1}, \beta_i\}}(e_j-l)\Big)\bigg(\sum_{\prod_{\substack{r=1\\r \notin K}}^ml_r}\mathcal{C}_{\prod_{\substack{r=1\\r \notin K}}^ml_r}(e_i)\prod_{\substack{r = 1\\r \notin K}}^{m}{e_r}^{l_r}\bigg),
	\end{multline*}
	where, $K = \{s: e_s \in M_2\} \cup \{i\}$ and $\mathcal{C}_{\prod_{\substack{r=1\\r \notin K}}^ml_r}(e_i)$ is a univariate polynomial in $e_i$.\\
\end{LEMMA}
Now, we have built all the necessary theory to define a formal algorithm (see page 12, Algorithm 1) that defines the steps of finding a $m$-tuple $(\beta_{e_1}, \beta_{e_2}, \ldots, \beta_{e_m})$, $\beta_{e_j}\in \mathcal{K}  = \{1, 2, \dots, \Delta+1, \alpha\}$, such that $\mathbf{Q}_m(\beta_{e_1}, \beta_{e_2}, \ldots, \beta_{e_m}) \not\equiv 0$ mod $p$. To start the flow in the algorithm, we have to prove the following Claim.\\
\begin{CLAIM}
	\label{cl1}
	There exists $(\beta_{e_1}, \beta_{e_2}, \ldots, \beta_{e_m})$, $\beta_{e_j}\in \mathcal{K} = \{1, 2, \dots, \Delta+1, \alpha = \Delta+2\}$, such that $\mathbf{Q'}_1(\beta_{e_1}, \beta_{e_2}, \ldots, \beta_{e_m}) \not\equiv 0$ mod $p$. That is, $\mathbf{Q'}_1(\boldsymbol{e}) \not\equiv 0$ mod $p$.
\end{CLAIM}

\begin{algorithm}
	\DontPrintSemicolon
	\SetAlgoLined
	{ 
		$i \leftarrow 1$, $\alpha_1=\Delta+2$.\\
		\While {$i \leq m$}
		{
			\BlankLine
			\If {there exists a point $(\beta_{e_1}, \beta_{e_2}, \ldots, \beta_{e_m})$ $(\beta_{e_j}\in \mathcal{K} = \{1, 2, \dots, \Delta+1, \alpha = \alpha_i\})$ such that $\mathbf{Q}_i(\beta_{e_1}, \beta_{e_2}, \ldots, \beta_{e_m}) \not\equiv 0$ mod $p$}
			{
				Go to Step 14.\\
			}
			\Else 
			{
				\tcc{\small{there does not exists a point $(\beta_{e_1}, \beta_{e_2}, \ldots, \beta_{e_m})$ $(\beta_{e_j}\in \mathcal{K} = \{1, 2, \dots, \Delta+1,$ $ \alpha = \alpha_{i-1}\})$ such that $\mathbf{Q}_i(\beta_{e_1}, \beta_{e_2}, \ldots, \beta_{e_m}) \not\equiv 0$ mod $p$. In other words, $\mathbf{Q'}_i(\boldsymbol{e}) \equiv 0$ mod $p$, is nothing but
						\begin{equation*}
						\mathcal{C}^{{\mathbf{P'}}}(\boldsymbol{e})\prod_{j=1}^{i}\Bigg(\prod_{\substack{\textnormal{if}\ N_j(e_j) \neq \emptyset \\ e_l \in N_j(e_j)}}({e_j} - e_l)\prod_{k \in \mathbb{Z}_p\setminus\{1, 2, \ldots, \Delta+1, \alpha_{i-1}\}}({e_j} - k)\Bigg)  \equiv 0\ \textnormal{mod}\ p
						\end{equation*}
				}}
				\small{Define,
					\begin{equation*}
					\mathbf{G}(\boldsymbol{e}) = \mathbf{Q}_{i-1}(\boldsymbol{e})\prod_{j=1}^{k}(e_{l_j}-\alpha_{i-1}),
					\end{equation*}
					\begin{equation*}
					\text{For every}\ e_j \in M_2, \text{define}\ \mathbf{H}_{e_j}(\boldsymbol{e}) = \mathbf{G'}(\boldsymbol{e})(e_j-\alpha_{i-1}).
					\end{equation*}}\\
				\tcc{\small{where $e_{l_1}, e_{l_2}, \ldots, e_{l_k} \in \{e_1, e_2, \ldots, e_i\}$ such that 
						$\mathbf{G'}(\boldsymbol{e}) \not\equiv 0$ mod $p$ and for every $e_j \in M_2$ the polynomial $\mathbf{H'}_{e_j}(\boldsymbol{e}) \equiv 0$ mod $p$}}
				
				Find $\beta_{i}\in \mathbb{Z}_p\setminus\{0, 1, 2, \ldots, \Delta+1, \alpha =\alpha_i\}$ such that for each $e_j \in M_2$, $(e_j-\beta_i)$ divides $\mathbf{G'}(\boldsymbol{e})$ 
				but $(e_j-\beta_i)^2$ does not divide $\mathbf{G'}(\boldsymbol{e})$. And also $e_i - \beta_i$ does not divide $\mathbf{G'}(\boldsymbol{e})$.\\
				\tcc{for each $e_j \in M_2$, $(e_j-\beta_i)$ is a  square-free factor in $\mathbf{G'}(\boldsymbol{e})$}
				\small{\begin{multline*}
					\mathbf{K}(\boldsymbol{e}) = \Bigg(\Bigg(\frac{\mathbf{G'}(\boldsymbol{e})}{\prod_{e_j \in M_2}(e_{j}-\beta_i)}\Bigg)\prod_{e_j \in M_2}\Big(e_{j}-\alpha_{i-1}\Big)\Bigg)\prod_{l\in \mathbb{Z}_p\setminus\{ 1, 2, \ldots, \Delta+1, \alpha=\beta_i\}}(e_i-l).\\
					\end{multline*}}\\
				\tcc{\small{for each $e_j \in M_2$, the square-free factor $(e_j-\beta_i)$ in $\mathbf{G'}(\boldsymbol{e})$  is replaced by $\big(e_j- \alpha_{i-1}\big)$}}
				There exists $(\beta'_{e_1}, \beta'_{e_2}, \ldots, \beta'_{e_{i-1}}, \beta'_{e_i}, e_{i+1}, \ldots, e_m)$ 
				$(\beta'_{e_j}\in \mathcal{K} = \{1, 2, \dots, \Delta+1, \alpha = \beta_i\})$ such that $\mathbf{K}(\beta'_{e_1}, \beta'_{e_2}, \ldots, \beta'_{e_{i-1}}, \beta'_{e_i}, e_{i+1}, \ldots, e_m) \not\equiv 0$ \textnormal{mod} $p$.\\
				
				There exists $(\beta'_{e_1}, \beta'_{e_2}, \ldots, \beta'_{e_{i-1}}, \beta'_{e_i}, e_{i+1}, \ldots, e_m)$
				$(\beta'_{e_j}\in \mathcal{K} = \{1, 2, \dots, \Delta+1, \alpha = \beta_i\})$ such that $\mathbf{Q}_i(\beta'_{e_1}, \beta'_{e_2}, \ldots, \beta'_{e_{i-1}}, \beta'_{e_i}, e_{i+1}, \ldots, e_m) \not\equiv 0$ mod $p$.\\
				\tcc{\small{In other words,
						\begin{equation*}
						\mathcal{C}^{{\mathbf{P'}}}(\boldsymbol{e})\prod_{j=1}^{i}\Bigg(\prod_{\substack{\textnormal{if}\ N_j(e_j) \neq \emptyset \\ e_l \in N_j(e_j)}}({e_j} - e_l)\prod_{k \in \mathbb{Z}_p\setminus\{1, 2, \ldots, \Delta+1, \beta_i\}}({e_j} - k)\Bigg) \not\equiv 0\ \textnormal{mod}\ p
						\end{equation*}
				}}
				$\alpha_i \leftarrow \beta_{i}$.\\
				}
		$i \leftarrow i+1$.\\
		$\alpha_i \leftarrow \alpha_{i-1}$.\\
		}
		\textbf{Stop}
	}
	\caption{ \small{The algorithm defines the steps of finding a $m$-tuple $(\beta_{e_1}, \beta_{e_2}, \ldots, \beta_{e_m})$ $(\beta_{e_j}\in \mathcal{K}  = \{1, 2, \dots, \Delta+1, \alpha\})$ such that $\mathbf{Q}_m(\beta_{e_1}, \beta_{e_2}, \ldots, \beta_{e_m}) \not\equiv 0$ mod $p$. Without loss of generality, we assume $\alpha_1 = \Delta+2$.}}
\end{algorithm}

\setlength{\textfloatsep}{1pt}

Actually, by proving the Claim \ref{cl1}, Claim \ref{cl2} and Claim \ref{cl3}, we have established the following result.\\
\begin{RESULT}
	\label{cl4}
	There exists $(\beta_{e_1}, \beta_{e_2}, \ldots, \beta_{e_m})$, $\beta_{e_j}\in \mathcal{K} = \{1, 2, \dots, \Delta+1, \alpha \}$ $(1\leq j \leq m)$, such that $\mathbf{Q}_m(\beta_{e_1}, \beta_{e_2}, \ldots, \beta_{e_m}) \not\equiv 0$ mod $p$. That is,\\
	\begin{multline*}
	\mathcal{C}^{{\mathbf{P'}}}(\boldsymbol{e})\prod_{j=1}^{m}\Bigg(\prod_{\substack{\textnormal{if}\ N_j(e_j) \neq \emptyset \\ e_l \in N_j(e_j)}}({e_j} - e_l)\prod_{k \in \mathbb{Z}_p\setminus\{1, 2, \ldots, \Delta+1,\alpha\}}({e_j} - k)\Bigg)  \not\equiv 0\ \textnormal{mod}\ p.\\
	\end{multline*}
\end{RESULT}

Finally, to establish that for a given graph,  the \textit{total chromatic number} $\chi''(G)$ is bounded above by $\Delta +2$, we prove the following theorem,

\begin{THEOREM}
	\label{pt4}
	There exists $(\beta_{e_1}, \beta_{e_2}, \ldots,$ $\beta_{e_m})$, $\beta_{e_i} \in \{1, 2, \ldots, \Delta +1, \alpha\}$ $(1\leq i \leq m)$, such that $\mathcal{C}^{\mathbf{P'}}(\beta_{e_1}, \beta_{e_2}, \ldots,$ $\beta_{e_m}) \not\equiv 0$ mod $p$, and $(\beta_{v_1}, \beta_{v_2}, \ldots, \beta_{v_n})$, $\beta_{v_i} \in \{1, 2, \ldots, \Delta +1\}$ $(1\leq i \leq n)$, such that $\mathbf{P'}(\beta_{v_1}, \beta_{v_2}, \ldots, \beta_{v_n}, \beta_{e_1}, \beta_{e_2}, \ldots, \beta_{e_m}) \not\equiv 0$ mod $p$.
\end{THEOREM}
The Theorem \ref{pt4} will have the following corollary. 
\begin{COROLLARY}
	\label{pc1}
	For any graph $G$, $\chi''(G)\leq \Delta + 2$.
\end{COROLLARY}

Theorem \ref{t1} will guarantee that $\mathbf{P'}(v_1, v_2, \ldots, v_n, e_1, e_2, \ldots, e_m) \not\equiv 0$ mod $p$. That is, $\mathbf{P'}(v_1, v_2, \ldots, v_n, e_1, e_2, \ldots, e_m)$ is not a zero polynomial. Result \ref{cl4} not only guarantee that all the edges of the given graph are properly colored using $\Delta+2$ colors $(\beta_{e_i} \in \{1, 2, \ldots, \Delta +1, \alpha\})$, it also proves that $\mathcal{C}^{\mathbf{P'}}(\beta_{e_1}, \beta_{e_2}, \ldots,$ $\beta_{e_m}) \not\equiv 0$ mod $p$. Theorem \ref{pt4} and Corollary \ref{pc1} will guarantee that the  \textit{total chromatic number} $\chi''(G)$ of the given graph is bounded above by $\Delta + 2$ colors. That is, the mapping $f(v_i) = \beta_{v_i}$($1 \leq i \leq n$) and $f(e_i) = \beta_{e_i}$($1 \leq i \leq m$) will establish the desired result of proving the \textit{total chromatic number} $\chi''(G)$ of the given graph  to have a upper bound of $\Delta + 2$.\\

\section{Proofs of a remark, lemmas, claims, theorems, and a corollary}

\begin{proof} [Proof of Theorem \ref{t1}]
	Without loss of generality, let $e_i = 0$ (for all $i$, $1\leq i \leq m$) in $\mathbf{P}(v_1, v_2, \ldots, v_n, e_1, e_2, \ldots, e_m)$.
	By this we get, 
	
	\begin{equation*}
	\label{e1}
	\mathbf{P}(v_1, v_2, \ldots, v_n, 0, 0, \ldots, 0) = \prod_{i=1}^{n}\ \Big( \prod_{\substack{\textnormal{if}\ N_i(v_i) \neq \emptyset \\ v_j \in N_i(v_i)}}(v_i - v_j)\prod_{e_j \in N_e(v_i)}v_i\prod_{l = \Delta + 2}^{p}(v_i-l)\Big).
	\end{equation*}
	Here, finding $\alpha_{v_i} \in \{1, 2, \ldots, \Delta + 1\}$ such that 
	$\mathbf{P}(\alpha_{v_1}, \alpha_{v_2}, \ldots,\alpha_{v_n}, 0, 0, \ldots, 0) \not\equiv 0$ mod $p$ is nothing but obtaining the \textit{vertex coloring} of the given graph. That is, 
	the map $f(v_i) = \alpha_{v_i}$ ($1\leq i \leq n$) defines the \textit{vertex coloring} of the given graph.\\
	
	We already know from Brooks' theorem, $\chi(G) = \Delta + 1$, and this implies $\mathbf{P'}(v_1, v_2, \ldots,$ $v_n, 0, 0, \ldots, 0) \not\equiv 0$ mod $p$. We can also see, $\mathbf{P'}(v_1, v_2,$ $\ldots, v_n, e_1, e_2, \ldots, e_m) \not\equiv 0$ mod $p$. Therefore, there exists a coefficient $\mathcal{C}^{\mathbf{P'}}(\boldsymbol{e})$($\not\equiv 0$ mod $p$) of \ $\prod_{j=1}^{n}v_{j}^{l_j}$(\ for some \ $l_1\geq0, l_2\geq0, \ldots, l_n\geq0$) in $\mathbf{P'}(v_1, v_2, \ldots, v_n, e_1, e_2, \ldots, e_m)$.
\end{proof}

\begin{proof}[Proof of Lemma \ref{l1}]
	Given an edge $e_k = \{v_i, v_j\}$, $e_k$ is incident to vertices $v_i$ and $v_j$. Therefore, exponents of $e_k$'s ($1 \leq k \leq m$) in $\mathbf{P}(v_1, v_2, \ldots, v_n, e_1, e_2, \ldots, e_m)$ is always less than or equal to 2. So the same holds true for the coefficient $\mathcal{C}^{\mathbf{P'}}(\boldsymbol{e})$.
\end{proof}

\begin{proof} [Proof of Remark \ref{pl2}]
	Given an edge $e_i = \{v_s, v_t\}$, we have $\mathbf{S}_i = \{e_{l_1}, e_{l_2}, \ldots, e_{l_r} :e_{l_j} \in N_e(v_s), 1 \leq l_1, l_2, \dots, l_r \leq m\}$ and $|\mathbf{S}_i| \leq r \leq \Delta$. And we have\\	
	\begin{equation*}
	\mathbf{Z}_i(\boldsymbol{e}) = \mathcal{C}^{\mathbf{P'}}(\boldsymbol{e}) \prod_{1 \leq j < k \leq r}(e_{l_j} - e_{l_k})
	\prod_{e_{l_j} \in \mathbf{S}_i}(\prod_{l = \Delta + 3}^p(e_{l_j} - l)).
	\end{equation*}
	
	By Lemma \ref{l1}, exponent of each variable in $\mathcal{C}^{\mathbf{P'}}(\boldsymbol{e})$ is $\leq 2$. And exponent of each variable in  $$\prod_{1 \leq j < k \leq r}(e_{l_j} - e_{l_k})
	\prod_{e_{l_j} \in \mathbf{S}_i}(\prod_{l = \Delta + 3}^p(e_{l_j} - l))$$ is $\leq p-3$. Therefore, there always exists a monomial $\prod_{j=1}^r e_{l_j}^{s_j}$ (for some $s_j \geq 0$) in $\mathbf{Z}_i(\boldsymbol{e})$ whose coefficient is $\not\equiv 0$ mod $p$, this implies  $\mathbf{Z}_i(\boldsymbol{e}) \not\equiv 0$ mod $p$.	
\end{proof}

\begin{proof} [Proof of Claim \ref{cl1}]
	Without loss of generality, we assume $\alpha_1 = \Delta+2$. 
	By Lemma \ref{l1}, exponent of each variable $e_j$ in $\mathcal{C}^{\mathbf{P'}}(\boldsymbol{e})$ is $\leq 2$. $\mathcal{C}^{\mathbf{P'}}(\boldsymbol{e})$ can be rewritten as, $$\mathcal{C}^{\mathbf{P'}}(\boldsymbol{e}) = \sum_{j = 0}^{2}a_j(e_{2}, \ldots, e_m)e_1^j,$$ where $a_j(e_{2}, e_{3}, \ldots, e_m)$ is either a polynomial 
	in $e_{2}, e_{3}, \ldots, e_m$ or constant.\\
	To the variable $e_1$, we associate the set $\mathcal{A}_1 = \mathcal{K} = \{1, 2, \ldots, \Delta+1, \alpha = \Delta+2\}$. Then there exists $\beta_{e_1} \in \mathcal{A}_1$ such that $\mathcal{C}^{\mathbf{P'}}(\beta_{e_1}, e_{2}, \ldots, e_m) \not\equiv 0$ mod $p$.\\
	
	We get, $$\mathbf{Q}_{1}(\beta_{e_1}, e_2, \ldots, e_m) = \mathcal{C}^{\mathbf{P'}}(\beta_{e_1}, e_{2}, \ldots, e_m)\prod_{\substack{\textnormal{if}\ N_1(e_1) \neq \emptyset \\ e_j \in N_1(e_1)}}(\beta_{e_1} - e_j)\prod_{l \in \mathbb{Z}_p\setminus \{1, 2, \ldots, \Delta +1, \Delta +2\}}(\beta_{e_1} - l).$$
	Let $\prod_{j = 2}^ me_j^{l_j}$ ($0 \leq l_{j} \leq 2$) be a monomial of maximum degree in $\mathcal{C}^{\mathbf{P'}}(\beta_{e_1}, e_{2}, \ldots, e_m)$, and we have that  
	$\prod_{ e_j \in N_1(e_1)}e_j$  is the only monomial (unique) of maximum degree in $$\prod_{\substack{\textnormal{if}\ N_1(e_1) \neq \emptyset \\ e_j \in N_1(e_1)}}(\beta_{e_1} - e_j).$$ 
	The product of the above two monomials give a unique monomial of maximum degree in 
	$\mathbf{Q}_{1}(\beta_{e_1}, e_2, \ldots, e_m)$.\\
	
	So,  $\mathbf{Q}_{1}(\beta_{e_1}, e_2, \ldots, e_m) \not\equiv 0$ mod $p$. This implies, $\mathbf{Q'}_{1}({e_1}, e_2, \ldots, e_m) \not\equiv 0$ mod $p$.
	
\end{proof}
\begin{proof}[Proof of Lemma \ref{l2}]
	Given an $e_j \in M_2$, from the relation (\ref{eq11}) we have,
	\begin{equation*}
	\mathbf{G'}(\boldsymbol{e}) \equiv \sum_{\prod_{\substack{r=1\\r\neq j}}^ml_r}\mathcal{C}_{\prod_{\substack{r=1\\r\neq j}}^ml_r}(e_j)\prod_{\substack{r = 1\\r \neq j}}^{m}{e_r}^{l_r},
	\end{equation*}
	where $\mathcal{C}_{\prod_{\substack{r=1\\r\neq j}}^ml_r}(e_j)$ is the coefficient of $\prod_{\substack{r = 1\\r \neq j}}^{m}{e_r}^{l_r}$, and is a univariate polynomial in $e_j$, and exponent of $e_j$ and each $e_r$ $(\substack{1 \leq r \leq m \\r\neq j})$ is $\leq p-1$.\\ 
	
	So, given an $e_j \in M_2$, using the relation (\ref{eq11}), $\mathbf{H}_{e_j}(\boldsymbol{e}) = \mathbf{G'}(\boldsymbol{e})(e_j-\alpha_{i-1})$ can be rewritten as, 
	\begin{equation}
	\label{m6}
	\mathbf{H}_{e_j}(\boldsymbol{e}) \equiv \Big( \sum_{\prod_{\substack{r=1\\r\neq j}}^ml_r}\mathcal{C}_{\prod_{\substack{r=1\\r\neq j}}^ml_r}(e_j)\prod_{\substack{r = 1\\r \neq j}}^{m}{e_r}^{l_r}\Big)\big(e_j-\alpha_{i-1}\big).
	\end{equation}\\

	Let $\mathbf{L}_{{\prod_{\substack{r =1 \\ r \neq j}}^m l_r}}(e_{j}) = \mathcal{C}_{\prod_{\substack{r=1\\r\neq j}}^ml_r}(e_j)(e_j-\alpha_{i-1})$,\\
	where $\mathcal{C}_{\prod_{\substack{r=1\\r\neq j}}^ml_r}(e_j)$ is the coefficient of $\prod_{\substack{r = 1\\r \neq j}}^{m}{e_r}^{l_r}$ in $\mathbf{G'}(\boldsymbol{e}).$
	
	Given an $e_j \in M_2$, we have, 
	$\mathbf{H'}_{e_j}({e_1}, {e_{2}}, \ldots, {e_m}) \equiv 0$. In other words, we can say that, for every coefficient 
	$\mathcal{C}_{\prod_{\substack{r=1\\r\neq j}}^ml_r}(e_j)$ in $\mathbf{G'}(\boldsymbol{e})$ (congruence relation (\ref{eq11})),
	\begin{equation}
	\label{eq1}  
	\mathbf{L'}_{{\prod_{\substack{r =1 \\ r \neq j}}^m l_r}}(e_{j}) \equiv 0\ \textnormal{mod}\ p.
	\end{equation}
	\\
	
	Since, $\mathbf{L'}_{{\prod_{\substack{r =1 \\ r \neq j}}^m l_r}}(e_{j}) \equiv 0$ mod $p$, $\mathcal{C}_{\prod_{\substack{r=1\\r\neq j}}^ml_r}(e_j)$ is the coefficient of $\prod_{\substack{r = 1\\r \neq j}}^{m}{e_r}^{l_r}$ and exponent of $e_j$ is $\leq p-1$, we can conclude that, 
	\begin{equation}
	\label{eq2}
	\mathcal{C}_{\prod_{\substack{r=1\\r\neq j}}^ml_r}(e_j) \equiv b^j_{\prod_{\substack{r=1\\r\neq j}}^ml_r}\prod_{l \in \mathbb{Z}_p\setminus\{\alpha_{i-1}\}}(e_j-l),
	\end{equation}
	where $b^j_{\prod_{\substack{r=1\\r\neq j}}^ml_r} \in \mathbb{Z}_p\setminus\{0\}$.\\

	From the above congruence relations (\ref{eq1}) and (\ref{eq2}), we can rewrite the polynomial $\mathbf{G'}(\boldsymbol{e})$ as follows, 
	\begin{equation*}
	\mathbf{G'}(\boldsymbol{e}) \equiv \sum_{\prod_{\substack{r=1\\r\neq j}}^ml_r}\Bigg(b^j_{\prod_{\substack{r=1\\r\neq j}}^ml_r}\prod_{l \in \mathbb{Z}_p\setminus\{\alpha_{i-1}\}}(e_j-l)\Bigg)\prod_{\substack{r = 1\\r \neq j}}^{m}{e_r}^{l_r},
	\end{equation*}
	where $b^j_{\prod_{\substack{r=1\\r\neq j}}^ml_r} \in \mathbb{Z}_p\setminus\{0\}$.\\ 
	So, 
	\begin{equation}
	\label{m4}
	\mathbf{G'}(\boldsymbol{e}) \equiv \prod_{l \in \mathbb{Z}_p\setminus\{\alpha_{i-1}\}}(e_j-l)\bigg(\sum_{\prod_{\substack{r=1\\r\neq j}}^ml_r}b^j_{\prod_{\substack{r=1\\r\neq j}}^ml_r}\prod_{\substack{r = 1\\r \neq j}}^{m}{e_r}^{l_r}\bigg),
	\end{equation}
	where $b^j_{\prod_{\substack{r=1\\r\neq j}}^ml_r} \in \mathbb{Z}_p\setminus\{0\}$.\\ \\
	
	Since, for every $e_j \in M_2$, the polynomial $\mathbf{H'}_{e_j}(\boldsymbol{e}) \equiv 0$ mod $p$, we can rewrite the polynomial $\mathbf{G'}(\boldsymbol{e})$ as follows, 
	\begin{equation}
	\label{m5}
	\mathbf{G'}(\boldsymbol{e}) \equiv \prod_{e_j \in M_2}\Big(\prod_{l \in \mathbb{Z}_p\setminus\{\alpha_{i-1}\}}(e_j-l)\Big)\bigg(\sum_{\prod_{\substack{r=1\\r \notin K}}^ml_r}\mathcal{C}_{\prod_{\substack{r=1\\r \notin K}}^ml_r}(e_i)\prod_{\substack{r = 1\\r \notin K}}^{m}{e_r}^{l_r}\bigg),
	\end{equation}
	where, $K = \{s: e_s \in M_2\} \cup \{i\}$ and $\mathcal{C}_{\prod_{\substack{r=1\\r \notin K}}^ml_r}(e_i)$ is a univariate polynomial in $e_i$.\\
	
\end{proof}
\begin{proof}[Proof of Claim \ref{cl2} ]
	We have, $\mathcal{K} = \{1, 2, \ldots, \Delta, \Delta+1\} \cup \{\alpha = \alpha_{i-1}\}$ and we have to find a new value to $\alpha$.\\

	Now, to find a new value to $\alpha$, that is, $\alpha = \beta_i$ such that for each $e_j \in M_2$, $(e_j-\beta_i)$ divides
	$\mathbf{G'}(\boldsymbol{e})$ 
	but $(e_j-\beta_i)^2$ does not divide $\mathbf{G'}(\boldsymbol{e})$, we consider the polynomials $\{\mathbf{J}_{\prod_{\substack{r=1\\r\neq j}}^ml_r}(e_j):e_j \in M_2\} \cup \{\mathbf{J}_{\prod_{\substack{r=1\\r\neq i}}^ml_r}(e_i)\}$ (as defined in (\ref{eq3}) and (\ref{eq10})).\\
	
	For each $e_j \in M_2$, we have
	\begin{equation*}
	\mathbf{J}_{\prod_{\substack{r=1\\r\neq j}}^ml_r}(e_j) = \mathcal{C}^j_{\prod_{\substack{r=1\\r\neq j}}^ml_r}(e_j)\prod_{l\in \mathbb{Z}_p\setminus\{ 1, 2, \ldots, \Delta+1, \alpha_{i-1}\}}(e_j-l),
	\end{equation*}
	using the congruence relations (\ref{eq12}) and (\ref{eq2}), we get 
	\begin{equation}
	\label{eq13}
	\mathbf{J'}_{\prod_{\substack{r=1\\r\neq j}}^ml_r}(e_j) \equiv b^j_{\prod_{\substack{r=1\\r\neq j}}^ml_r}\prod_{l \in \mathbb{Z}_p\setminus\{\alpha_{i-1}\}}(e_j-l),
	\end{equation}
	where, $b^j_{\prod_{\substack{r=1\\r\neq j}}^ml_r} \in \mathbb{Z}_p\setminus\{0\}$.\\
	And we have (using the relation (\ref{m5})), 
	\begin{equation}
	\label{m12}
	\mathbf{J}_{\prod_{\substack{r=1\\r\neq i}}^ml_r}(e_i) = \mathcal{C}^i_{\prod_{\substack{r=1\\r\neq i}}^ml_r}(e_i) \equiv \mathcal{C}_{\prod_{\substack{r=1\\r \notin K}}^ml_r}(e_i),
	\end{equation}
	where, $\mathcal{C}_{\prod_{\substack{r=1\\r \notin K}}^ml_r}(e_i)$ is the coefficient of some monomial $\prod_{\substack{r = 1\\r \notin K}}^{m}{e_r}^{l_r}$ (from the congruence relation (\ref{m5})).
		
	Now, we involve in defining a subset of $\mathbb{Z}_p\setminus\{0, 1, 2, \ldots, \Delta+1, \alpha_{i-1}\}$ from which $\beta_i$ is chosen.\\
	
	Let ${\mathcal{B}} = \mathbb{Z}_p\setminus\{0, 1, 2, \ldots, \Delta+1, \alpha_{i-1}\}$. 
	
	Since $p\geq m^2(2\Delta+2)$ and $|{\mathcal{B}}|$ is greater than $2m(\Delta+1)$, there exists $\beta_i \in \mathcal{B}$ such that for each $e_j \in M_2$, $e_j - \beta_i$ divides $\mathbf{J}_{\prod_{\substack{r=1\\r\neq j}}^ml_r}(e_j)$ but $(e_j - \beta_i)^2$ does not divide $\mathbf{J}_{\prod_{\substack{r=1\\r\neq j}}^ml_r}(e_j)$. And how to choose such a $\beta_i$ is 
	explained below.\\
	
	For each $e_j \in M_2$,\\
	let $\mathcal{B}_j = \{\gamma \in \mathcal{B}: (e_j - \gamma)^2\ divides\ \mathbf{J}_{\prod_{\substack{r=1\\r\neq j}}^ml_r}(e_j) \}\subset \mathcal{B}$. As noted earlier, for each $e_j \in M_2 $, 
	exponent of $e_j$ in $\mathbf{J}_{\prod_{\substack{r=1\\r\neq j}}^ml_r}(e_j)$ \textnormal{is always} $\leq (2\Delta+p-(\Delta+2)) = p+\Delta-2$ and congruence relation (\ref{eq13}) implies number of repeated roots can be at most $\Delta-1$. Then cardinality of $\mathcal{B}_j$ is at most $\Delta-1$. 
	And let $\mathcal{B}_i = \{\gamma \in \mathcal{B}: e_i - \gamma\ divides\ \mathcal{C}_{\prod_{\substack{r=1\\r \notin K}}^ml_r}(e_i) \}$, then $|\mathcal{B}_i|\leq 2\Delta$. \\
	
	Now, we choose a $\beta_i$ from the set $\mathcal{B}\setminus \cup_{k=1}^i \mathcal{B}_j $, that is, $\beta_i \in \mathcal{B}\setminus \cup_{k=1}^i \mathcal{B}_j $. \\ 
	
	So, for each $e_j \in M_2 \cup \{e_i\}$, $(e_j-\beta_i)$ divides $\mathbf{J}_{\prod_{\substack{r=1\\r\neq j}}^ml_r}(e_j)$ but $(e_j-\beta_i)^2$ does not divide $\mathbf{J}_{\prod_{\substack{r=1\\r\neq j}}^ml_r}(e_j)$.\\
	
	From the congruence relations (\ref{m3}) and (\ref{eq12}), we can conclude that, for each $e_j \in M_2$, $(e_j-\beta_i)$ divides
	$\mathbf{G'}(\boldsymbol{e})$ 
	but $(e_j-\beta_i)^2$ does not divide $\mathbf{G'}(\boldsymbol{e})$. From the congruence relation (\ref{m12}) and the choice of $\beta_i$, we observe that $e_i - \beta_i$ does not divide the $\mathcal{C}_{\prod_{\substack{r=1\\r \notin K}}^ml_r}(e_i)$, so $e_i - \beta_i$ does not divide $\mathbf{G'}(\boldsymbol{e})$.
\end{proof} 

\begin{proof} [Proof of Lemma \ref{l3}]
	By the definition of  $\mathbf{G}(\boldsymbol{e})$, we have \\
	
	{ 
		\begin{multline*}
		\mathbf{G}(\boldsymbol{e}) = \mathcal{C}^{{\mathbf{P'}}}(\boldsymbol{e})\Bigg(\prod_{j=1}^{i-1}\Big(\prod_{\substack{\textnormal{if}\ N_j(e_j) \neq \emptyset \\ e_l \in N_j(e_j)}}({e_j} - e_l)\prod_{k \in \mathbb{Z}_p\setminus\{1, 2, \ldots, \Delta+1, \alpha = \alpha_{i-1}\}}({e_j} - k)\Big)\Bigg)\Bigg(\prod_{j=1}^{k}(e_{l_j}-\alpha_{i-1})\Bigg).
		\end{multline*}}
	
	After rearranging factors, $\mathbf{G}(\boldsymbol{e})$ can be rewritten as,
	
	\begin{multline}
	\label{m9}
	\mathbf{G}(\boldsymbol{e}) = \mathcal{C}^{\mathbf{P'}}(\boldsymbol{e})\Bigg( \prod_{e_j \in M_1} \Big(\prod_{\substack{\textnormal{if}\ N_j(e_j) \neq \emptyset \\ e_l \in N_j(e_j)}}({e_j} - e_l)\prod_{k \in \mathbb{Z}_p\setminus\{1, 2, \ldots, \Delta+1\}}({e_j} - k)\Big)\\\prod_{e_j \in M_2} \Big(\prod_{\substack{\textnormal{if}\ N_j(e_j) \neq \emptyset \\ e_l \in N_j(e_j)}}({e_j} - e_l)\prod_{k \in \mathbb{Z}_p\setminus\{1, 2, \ldots, \Delta+1, \alpha_{i-1}, \beta_i\}}({e_j} - k)\Big)\Bigg)\prod_{e_j \in M_2}\big(e_j-\beta_i\big).
	\end{multline}
	
	From the congruence relation (\ref{m5}), we have 
	
	\begin{equation*}
	\mathbf{G'}(\boldsymbol{e}) \equiv \prod_{e_j \in M_2}\Big(\prod_{l \in \mathbb{Z}_p\setminus\{\alpha_{i-1}\}}(e_j-l)\Big)\bigg(\sum_{\prod_{\substack{r=1\\r \notin K}}^ml_r}\mathcal{C}_{\prod_{\substack{r=1\\r \notin K}}^ml_r}(e_i)\prod_{\substack{r = 1\\r \notin K}}^{m}{e_r}^{l_r}\bigg),
	\end{equation*}
	where, $K = \{s: e_s \in M_2\} \cup \{i\}$ and $\mathcal{C}_{\prod_{\substack{r=1\\r \notin K}}^ml_r}(e_i)$ is a univariate polynomial in $e_i$.
	
	After rearranging factors, $\mathbf{G'}(\boldsymbol{e})$ can be rewritten as,
	
	\begin{multline}
	\label{m10}
	\mathbf{G'}(\boldsymbol{e}) \equiv \prod_{e_j \in M_2}\Big(\prod_{l \in \mathbb{Z}_p\setminus\{\alpha_{i-1}, \beta_i\}}(e_j-l)\Big)\Bigg(\sum_{\prod_{\substack{r=1\\r \notin K}}^ml_r}\mathcal{C}_{\prod_{\substack{r=1\\r \notin K}}^ml_r}(e_i)\prod_{\substack{r = 1\\r \notin K}}^{m}{e_r}^{l_r}\Bigg)\prod_{e_j \in M_2}\Big(e_j-\beta_i\Big),
	\end{multline}
	where, $K = \{s: e_s \in M_2\} \cup \{i\}$ and $\mathcal{C}_{\prod_{\substack{r=1\\r \notin K}}^ml_r}(e_i)$ is a univariate polynomial in $e_i$.
	
	Now, we consider $\frac{\mathbf{G}(\boldsymbol{e})}{\prod_{e_j \in M_2}\Big(e_j-\beta_i\Big)}$ and $\frac{\mathbf{G'}(\boldsymbol{e})}{\prod_{e_j \in M_2}\Big(e_j-\beta_i\Big)}$.\\ \\
	
	From the relation (\ref{m9}), we have,
	 
	\begin{multline*} 
	\frac{\mathbf{G}(\boldsymbol{e})}{\prod_{e_j \in M_2}\Big(e_j-\beta_i\Big)} = \mathcal{C}^{\mathbf{P'}}(\boldsymbol{e})\Bigg(\prod_{e_j \in M_1} \Big(\prod_{\substack{\textnormal{if}\ N_j(e_j) \neq \emptyset \\ e_l \in N_j(e_j)}}({e_j} - e_l)\prod_{k \in \mathbb{Z}_p\setminus\{1, 2, \ldots, \Delta+1\}}({e_j} - k)\Big)\\\prod_{e_j \in M_2} \Big(\prod_{\substack{\textnormal{if}\ N_j(e_j) \neq \emptyset \\ e_l \in N_j(e_j)}}({e_j} - e_l)\prod_{k \in \mathbb{Z}_p\setminus\{1, 2, \ldots, \Delta+1, \alpha_{i-1}, \beta_i\}}({e_j} - k)\Big)\Bigg).
	\end{multline*}
	
	And from the relation (\ref{m10}), we have, 
	\begin{multline}
	\label{m11}
	\frac{\mathbf{G'}(\boldsymbol{e})}{\prod_{e_j \in M_2}\Big(e_j-\beta_i\Big)} \equiv \prod_{e_j \in M_2}\Big(\prod_{l \in \mathbb{Z}_p\setminus\{\alpha_{i-1}, \beta_i\}}(e_j-l)\Big)\bigg(\sum_{\prod_{\substack{r=1\\r \notin K}}^ml_r}\mathcal{C}_{\prod_{\substack{r=1\\r \notin K}}^ml_r}(e_i)\prod_{\substack{r = 1\\r \notin K}}^{m}{e_r}^{l_r}\bigg),
	\end{multline}
	where, $K = \{s: e_s \in M_2\} \cup \{i\}$ and $\mathcal{C}_{\prod_{\substack{r=1\\r \notin K}}^ml_r}(e_i)$ is a univariate polynomial in $e_i$.\\
	
	From the Claim \ref{cl2}, for each $e_j \in M_2 $, $(e_j-\beta_i)$ divides $\mathbf{J}_{\prod_{\substack{r=1\\r\neq j}}^ml_r}(e_j)$ but $(e_j-\beta_i)^2$ does not divide $\mathbf{J}_{\prod_{\substack{r=1\\r\neq j}}^ml_r}(e_j)$. Moreover, for each $e_j \in M_2$, $(e_j-\beta_i)$ divides
	$\mathbf{G'}(\boldsymbol{e})$ 
	but $(e_j-\beta_i)^2$ does not divide $\mathbf{G'}(\boldsymbol{e})$. And $e_i-\beta_i$ does not divide $\mathbf{G'}(\boldsymbol{e})$ as well. Therefore, we can conclude that,
	 
	\begin{equation*}
	\frac{\mathbf{G}(\boldsymbol{e})}{\prod_{e_j \in M_2}\Big(e_j-\beta_i\Big)} \equiv \frac{\mathbf{G'}(\boldsymbol{e})}{\prod_{e_j \in M_2}\Big(e_j-\beta_i\Big)},
	\end{equation*}
	that is,
	\begin{multline*}
	\mathcal{C}^{\mathbf{P'}}(\boldsymbol{e})\Bigg(\prod_{e_j \in M_1} \Big(\prod_{\substack{\textnormal{if}\ N_j(e_j) \neq \emptyset \\ e_l \in N_j(e_j)}}({e_j} - e_l)\prod_{k \in \mathbb{Z}_p\setminus\{1, 2, \ldots, \Delta+1\}}({e_j} - k)\Big) \\ \prod_{e_j \in M_2} \Big(\prod_{\substack{\textnormal{if}\ N_j(e_j) \neq \emptyset \\ e_l \in N_j(e_j)}}({e_j} - e_l)\prod_{k \in \mathbb{Z}_p\setminus\{1, 2, \ldots, \Delta+1, \alpha_{i-1}, \beta_i\}}({e_j} - k)\Big)\Bigg)\\ \equiv \prod_{e_j \in M_2}\Big(\prod_{l \in \mathbb{Z}_p\setminus\{\alpha_{i-1}, \beta_i\}}(e_j-l)\Big)\bigg(\sum_{\prod_{\substack{r=1\\r \notin K}}^ml_r}\mathcal{C}_{\prod_{\substack{r=1\\r \notin K}}^ml_r}(e_i)\prod_{\substack{r = 1\\r \notin K}}^{m}{e_r}^{l_r}\bigg),
	\end{multline*}
	where, $K = \{s: e_s \in M_2\} \cup \{i\}$ and $\mathcal{C}_{\prod_{\substack{r=1\\r \notin K}}^ml_r}(e_i)$ is a univariate polynomial in $e_i$.
\end{proof}

\begin{proof} [Proof of Claim \ref{cl3}] 
	Let us consider a polynomial $\mathbf{K}(\boldsymbol{e})$ as follows,
	
	\begin{equation*}
	\mathbf{K}(\boldsymbol{e}) =\Bigg(\frac{\mathbf{G'}(\boldsymbol{e})}{\prod_{e_j \in M_2}(e_j-\beta_i)}\Bigg)\prod_{e_j \in M_2}\Big(e_j-\alpha_{i-1}\Big)\prod_{l\in \mathbb{Z}_p\setminus\{ 1, 2, \ldots, \Delta+1, \alpha=\beta_i\}}(e_i-l),
	\end{equation*}
	
	using congruence relation (\ref{m11}), the polynomial $\mathbf{K}(\boldsymbol{e})$ can be rewritten as, \\
	
	\begin{multline*}
	\mathbf{K}(\boldsymbol{e}) \equiv  
	\prod_{e_j \in M_2}\Bigg(\Big(\prod_{l \in \mathbb{Z}_p\setminus\{\alpha_{i-1}, \beta_i\}}(e_j-l)\Big)\Big(e_j-\alpha_{i-1}\Big)\Bigg)\bigg(\sum_{\prod_{\substack{r=1\\r \notin K}}^ml_r}\Big(\mathcal{C}_{\prod_{\substack{r=1\\r \notin K}}^ml_r}(e_i)\prod_{l\in \mathbb{Z}_p\setminus\{ 1, 2, \ldots, \Delta+1, \alpha=\beta_i\}}(e_i-l)\Big)\prod_{\substack{r = 1\\r \notin K}}^{m}{e_r}^{l_r}\bigg).
	\end{multline*}\\

	Now, we can conclude that $\mathbf{K'}(\boldsymbol{e}) \not\equiv 0$ mod $p$ as explained below. The exponent of each $e_j$  in the following product,\\
	\begin{multline}
	\label{e3}
	\prod_{e_j \in M_2}\Bigg(\prod_{l \in \mathbb{Z}_p\setminus\{\alpha_{i-1}, \beta_i\}}\big(e_j-l\big)\big(e_j-\alpha_{i-1}\big)\Bigg) = \prod_{e_j \in M_2}\Bigg(\prod_{l \in \mathbb{Z}_p\setminus\{\beta_i\}}(e_j-l)\Bigg),
	\end{multline}
	is $\leq p-1$.\\
	
	From Claim \ref{cl2}, $e_i - \beta_i$ does not divide $\mathbf{G'}(\boldsymbol{e})$ and congruence relation (\ref{m12}) guarantee that existence of a $\mathcal{C}_{\prod_{\substack{r=1\\r \notin K}}^ml_r}(e_i)$ in $\mathbf{G'}(\boldsymbol{e})$ such that $e_i - \beta_i$ does not divide $\mathcal{C}_{\prod_{\substack{r=1\\r \notin K}}^ml_r}(e_i)$. Therefore, the following polynomial is not a zero polynomial, after applying Fermat's theorem to $e_i$, that is,\\
	 
	\begin{equation}
	\label{e4}
	\sum_{\prod_{\substack{r=1\\r \notin K}}^ml_r}\Big(\mathcal{C}_{\prod_{\substack{r=1\\r \notin K}}^ml_r}(e_i)\prod_{l\in \mathbb{Z}_p\setminus\{ 1, 2, \ldots, \Delta+1, \alpha=\beta_i\}}(e_i-l)\Big)\prod_{\substack{r = 1\\r \notin K}}^{m}{e_r}^{l_r} \not\equiv 0\ \textnormal{mod}\ p,
	\end{equation}
	and the  exponent of each $e_r$ ($r\neq i$) and exponent of $e_i$ (after applying Fermat's theorem)
	is $\leq p-1$.\\
	
	So, $\mathbf{K'}(\boldsymbol{e}) \not\equiv 0$ mod $p$, as $\{e_r:r \notin K, 1 \leq r \leq m\} \cap \{e_r:r \in K\} = \emptyset$ and $\mathbf{K'}(\boldsymbol{e})$ is the product of above polynomials (\ref{e3}) and (\ref{e4}).\\
	
	Since $\mathbf{K'}(\boldsymbol{e}) \not\equiv 0$ mod $p$, to each of the variable $e_1, e_{2}, \ldots, e_{m}$, we associate the sets $\mathcal{A}_1 = \mathbb{Z}_p, \mathcal{A}_{2} = \mathbb{Z}_p, \ldots,$ $\mathcal{A}_{m} = \mathbb{Z}_p$ respectively. Then there exists $\beta'_{e_1} \in \mathcal{A}_1, \beta'_{e_2} \in \mathcal{A}_{2}, \ldots,$ $\beta'_{e_m} \in \mathcal{A}_{m}$ such that
	
	\begin{equation}
	\label{eq4}
	\mathbf{K'}(\beta'_{e_1}, \beta'_{e_2}, \ldots, \beta'_{e_i}, \ldots, \beta'_{e_m}) \not\equiv 0\ \textnormal{mod}\ p.
	\end{equation}\\
	
	From the above congruence relation (\ref{eq4}), we can also conclude that $\mathbf{K'}(\boldsymbol{e}) \not\equiv 0$ mod $p$, and $\mathbf{K'}(\boldsymbol{e})$ can be rewritten as,\\
	
	\begin{multline*}
	\mathbf{K'}(\boldsymbol{e}) \equiv \prod_{e_j \in M_2}\Big(e_j-\alpha_{i-1}\Big)\Bigg(\prod_{e_j \in M_2}\bigg(\prod_{l \in \mathbb{Z}_p\setminus\{\alpha_{i-1}, \beta_i\}}(e_j-l)\bigg)\bigg(\sum_{\prod_{\substack{r=1\\r \notin K}}^ml_r}\Big(\mathcal{C}_{\prod_{\substack{r=1\\r \notin K}}^ml_r}(e_i)\Big)\prod_{\substack{r = 1\\r \notin K}}^{m}{e_r}^{l_r}\bigg)\Bigg) \\ \prod_{l\in \mathbb{Z}_p\setminus\{ 1, 2, \ldots, \Delta+1, \alpha=\beta_i\}}(e_i-l).
	\end{multline*}\\
	
	And from the Lemma \ref{l3} $\mathbf{K'}(\boldsymbol{e})$ can be rewritten as,
	\begin{multline*}
	\mathbf{K}(\boldsymbol{e}) = \prod_{e_j \in M_2}\Big(e_j-\alpha_{i-1}\Big)\Bigg(\mathcal{C}^{\mathbf{P'}}(\boldsymbol{e})\prod_{e_j \in M_1} \Big(\prod_{\substack{\textnormal{if}\ N_j(e_j) \neq \emptyset \\ e_l \in N_j(e_j)}}({e_j} - e_l)\prod_{k \in \mathbb{Z}_p\setminus\{1, 2, \ldots, \Delta+1\}}({e_j} - k)\Big)\\\prod_{e_j \in M_2} \Big(\prod_{\substack{\textnormal{if}\ N_j(e_j) \neq \emptyset \\ e_l \in N_j(e_j)}}({e_j} - e_l)\prod_{k \in \mathbb{Z}_p\setminus\{1, 2, \ldots, \Delta+1, \alpha_{i-1}, \beta_i\}}({e_j} - k)\Big)\Bigg) \prod_{k\in \mathbb{Z}_p\setminus\{ 1, 2, \ldots, \Delta+1, \alpha=\beta_i\}}(e_i-k).
	\end{multline*}\\
	So, 
	\begin{multline}
	\label{e5}
	\mathbf{K}(\boldsymbol{e}) \equiv \Bigg(\mathcal{C}^{\mathbf{P'}}(\boldsymbol{e})\prod_{j=1}^{i-1}\Big(\prod_{\substack{\textnormal{if}\ N_j(e_j) \neq \emptyset \\ e_l \in N_j(e_j)}}({e_j} - e_l)\prod_{k \in \mathbb{Z}_p\setminus\{1, 2, \ldots, \Delta+1, \beta_i\}}({e_j} - k)\Big) \Bigg) \prod_{k\in \mathbb{Z}_p\setminus\{ 1, 2, \ldots, \Delta+1, \alpha=\beta_i\}}(e_i-k).
	\end{multline}
	In other words, $$\mathbf{K}(\boldsymbol{e}) \equiv \mathbf{Q}_{i-1}(\boldsymbol{e})\prod_{k\in \mathbb{Z}_p\setminus\{ 1, 2, \ldots, \Delta+1, \alpha=\beta_i\}}(e_i-k) \not\equiv 0\ \textnormal{mod}\ p.$$ \\
	
	From the congruence relation (\ref{eq4}), we have $\mathbf{K}(\beta'_{e_1}, \beta'_{e_2}, \ldots, \beta'_{e_{i-1}}, \beta'_{e_i}, e_{i+1}, \ldots, e_m) \not\equiv 0$ \textnormal{mod} $p$. The following 
	products in $\mathbf{K}(\boldsymbol{e})$ make sure that $\beta_{e_j}\in \mathcal{K} = \{1, 2, \dots, \Delta+1, \alpha = \beta_i\}$
	$$\prod_{j=1}^i\Big(\prod_{k \in \mathbb{Z}_p\setminus\{1, 2, \ldots, \Delta+1, \beta_i\}}({e_j} - k)\Big).$$ \\ 
	
	We can see that the exponent of each variable $e_j$ ($j>i$) in the polynomial $\mathbf{K}(\beta'_{e_1}, \ldots, \beta'_{e_{i-1}}$, $\beta'_{e_i}, e_{i+1}, \ldots, e_m)$ is always less than or equal to $2\Delta$. Let $\prod_{j = i+1}^ me_j^{l_j}$ ($0 \leq l_{j} \leq 2\Delta$) be a monomial of maximum degree in $\mathbf{K}(\beta'_{e_1}, \ldots, \beta'_{e_{i-1}}, \beta'_{e_i}, e_{i+1}, \ldots, e_m)$, and we have that $\prod_{e_l \in N_i(e_i)}e_l$  is the only monomial (unique) of maximum degree in $$\prod_{\substack{\textnormal{if}\ N_i(e_i) \neq \emptyset \\ e_l \in N_i(e_i)}}(\beta'_{e_i} - e_l).$$\\
	 
	The product of the above two monomials give a unique monomial of maximum degree whose coefficient is non-zero ($\not\equiv 0$ mod $p$) in the following product of polynomials (\ref{e6}). Therefore, the following product of polynomials is not a zero polynomial,
	\begin{equation}
	\label{e6}
	\mathbf{K}(\beta'_{e_1}, \beta'_{e_2}, \ldots, \beta'_{e_{i-1}}, \beta'_{e_i}, e_{i+1}, \ldots, e_m)\prod_{\substack{\textnormal{if}\ N_i(e_i) \neq \emptyset \\ e_l \in N_i(e_i)}}(\beta'_{e_i} - e_l) \not\equiv 0\ \textnormal{mod}\ p.
	\end{equation}
	\textnormal{So, we can also conclude that,}\
	\begin{equation*}
	\mathbf{K}(\boldsymbol{e})\prod_{\substack{\textnormal{if}\ N_i(e_i) \neq \emptyset \\ e_l \in N_i(e_i)}}({e_i} - e_l) \not\equiv 0\ \textnormal{mod}\ p.
	\end{equation*}  
	Replacing $\mathbf{K}(\boldsymbol{e})$ by the relation (\ref{e5}), we get
	\begin{multline*}
	\Bigg(\mathcal{C}^{\mathbf{P'}}(\boldsymbol{e})\prod_{j=1}^{i-1}\Big(\prod_{\substack{\textnormal{if}\ N_j(e_j) \neq \emptyset \\ e_l \in N_j(e_j)}}({e_j} - e_l)\prod_{k \in \mathbb{Z}_p\setminus\{1, 2, \ldots, \Delta+1, \beta_i\}}({e_j} - k)\Big) \Bigg)\prod_{\substack{\textnormal{if}\ N_i(e_i) \neq \emptyset \\ e_l \in N_i(e_i)}}({e_i} - e_l) \prod_{k\in \mathbb{Z}_p\setminus\{ 1, 2, \ldots, \Delta+1, \alpha=\beta_i\}}(e_i-k),
	\end{multline*}
	that is, $\mathbf{Q}_{i-1}(\boldsymbol{e})\mathbf{E}^i(e_i, e_{i+1}, \ldots, e_m)$,
	this is nothing but, $\mathbf{Q}_i(\boldsymbol{e})$.\\
	
	From the congruence relation (\ref{e6}), we can conclusively say that, 
	there exist $\beta'_{e_j} \in \mathcal{K} = \{1, 2, \dots, \Delta+1, \alpha = \beta_i\}$ such that $$\mathbf{Q'}_i(\beta'_{e_1}, \beta'_{e_2}, \ldots, \beta'_{e_{i-1}}, \beta'_{e_i}, e_{i+1}, \ldots, e_m) \not\equiv 0\ \textnormal{mod}\ p.$$ 
\end{proof}

\begin{proof} [Proof of Theorem \ref{pt4}]
	By the Result \ref{cl4}, we have $\mathbf{Q}_m(\beta_{e_1}, \beta_{e_2}, \ldots, \beta_{e_m}) \not\equiv 0$ mod $p$. And 
	from the earlier stated definition, we have 
	\begin{equation*}
	\mathbf{Q}_m(\beta_{e_1}, \beta_{e_2}, \ldots, \beta_{e_m}) = \mathcal{C}^{\mathbf{P'}}(\beta_{e_1}, \beta_{e_2}, \ldots, \beta_{e_m})\prod_{j=1}^{m}(\mathbf{E}^j(\beta_{e_j}, \beta_{e_{j+1}}, \ldots, \beta_{e_m})) \not\equiv 0\ \textnormal{mod}\ p.
	\end{equation*}
	This can also be written as
	\begin{equation*}
	\mathbf{Q}_m(\beta_{e_1}, \beta_{e_2}, \ldots, \beta_{e_m}) = \mathcal{C}^{\mathbf{P'}}(\beta_{e_1}, \beta_{e_2}, \ldots, \beta_{e_m})\mathbf{E}_m(\beta_{e_1}, \beta_{e_2}, \ldots, \beta_{e_m}) \not\equiv 0\ \textnormal{mod}\ p.
	\end{equation*}
	
	If $\mathbf{Q}_m(\beta_{e_1}, \beta_{e_2}, \ldots, \beta_{e_m}) \not\equiv 0\ \textnormal{mod}\ p$, the mapping $f(e_i) = \beta_{e_i}$($1 \leq i \leq m$), as explained earlier, defines the \textit{edge coloring} of the given graph by using $\Delta + 2$ colors. And we have that $\mathcal{C}^{\mathbf{P'}}(\beta_{e_1}, \beta_{e_2}, \ldots, \beta_{e_m}) \not\equiv 0$ mod $p$ which is the coefficient of $\prod_{i=1}^nv_i^{l_i}$ (for some $l_i \geq 0$) in $\mathbf{P'}(v_1, v_2, \ldots, v_n,$ $e_1, e_2, \ldots, e_m)$. Therefore, the polynomial  $\mathbf{P'}(v_1, v_2, \ldots, v_n, \beta_{e_1}, \beta_{e_2}$, $\ldots, \beta_{e_m}) \not\equiv 0$ mod $p$.\\
	
	Since $\mathbf{P'}(v_1, v_2, \ldots, v_n, \beta_{e_1}, \beta_{e_2}$, $\ldots, \beta_{e_m}) \not\equiv 0$ mod $p$, to each variable $v_1$, $v_2$, \ldots, $v_n$ we associate the set $\mathcal{A}_1 = \mathbb{Z}_p$, $\mathcal{A}_2 = \mathbb{Z}_p$, \ldots, $\mathcal{A}_n = \mathbb{Z}_p$ respectively. Then there exists $\beta_{v_1} \in \mathcal{A}_1$, $\beta_{v_2} \in \mathcal{A}_2$ \ldots, $\beta_{v_n} \in \mathcal{A}_n$ such that $\mathbf{P'}(\beta_{v_1}, \beta_{v_2}, \ldots, \beta_{v_n}, \beta_{e_1}, \beta_{e_2}, \ldots, \beta_{e_m}) \not\equiv 0$ mod $p$.\\
	
	We get $\mathbf{P}(\beta_{v_1}, \beta_{v_2}, \ldots, \beta_{v_n}, \beta_{e_1}, \beta_{e_2}, \ldots, \beta_{e_m}) \equiv \mathbf{P'}(\beta_{v_1}, \beta_{v_2}, \ldots, \beta_{v_n}, \beta_{e_1}, \beta_{e_2}, \ldots, \beta_{e_m})\not\equiv 0$ mod $p$. The following products in $\mathbf{P}(v_1, v_2, \ldots, v_n, e_1, e_2, \ldots, e_m)$ make sure that $\beta_{v_i} \in \{1, 2, \ldots, \Delta +1\}$
	\begin{equation*}
	\prod_{i=1}^{n}\Bigg( \prod_{l = \Delta + 2}^{p}(\beta_{v_i}-l)\Bigg).
	\end{equation*}
	
\end{proof}

\begin{proof} [Proof of  Corollary \ref{pc1}]
	By the Result \ref{cl4} the mapping $f(e_i) = \beta_{e_i}$($1 \leq i \leq m$) defines the \textit{edge coloring} of the given graph by using $\Delta + 2$ colors.  Further the previous Theorem \ref{pt4} make sure that 
	no two adjacent vertices have the same color and no  edge  has  the  same  color  as  one  of  its end vertices.\\
	
	Therefore, the mapping $f(v_i) = \beta_{v_i}$($1 \leq i \leq n$) and $f(e_i) = \beta_{e_i}$($1 \leq i \leq m$), will conclusively define that the \textit{total coloring} of the given graph can be achieved by using $\Delta + 2$ colors.
\end{proof}
\clearpage
\section{Schematic representation of presentation of the paper}

	\tikzstyle{decision} = [diamond, draw,  text width=5.0em, text centered, node distance=4cm,]   
	\tikzstyle{document} = [rectangle, draw,   
	text width=4.7em, text centered, node distance=2.5cm,]   
	\tikzstyle{block1} = [rectangle, draw,   
	text width=19.5em, text centered, rounded corners, node distance= 4cm, minimum height=3em] 
	\tikzstyle{block} = [rectangle, draw,    
	text width=13.5em, text centered, rounded corners, node distance=7cm, minimum height=4.5em]  
	\tikzstyle{block2} = [rectangle, draw,    
	text width=5em, text centered, rounded corners, node distance=3cm, minimum height=4em] 
	\tikzstyle{line} = [draw, -latex']  
	\tikzstyle{cloud} = [draw, ellipse,text width= 2.9em, node distance=2.3cm, minimum height=3em]  
	
	\tikzstyle{ioi} = [trapezium, draw, trapezium right angle=120,rounded corners,  node distance=7cm, minimum height=4.5em]  
	\tikzstyle{io} = [trapezium, draw, trapezium right angle=110,rounded corners,  node distance=4cm, minimum height=.6em]   

		\center{\textbf{The following flowchart gives us the outer view of the flow in the \lq Algorithm 1 \rq\ in obtaining Result \ref{cl4}}\\
		
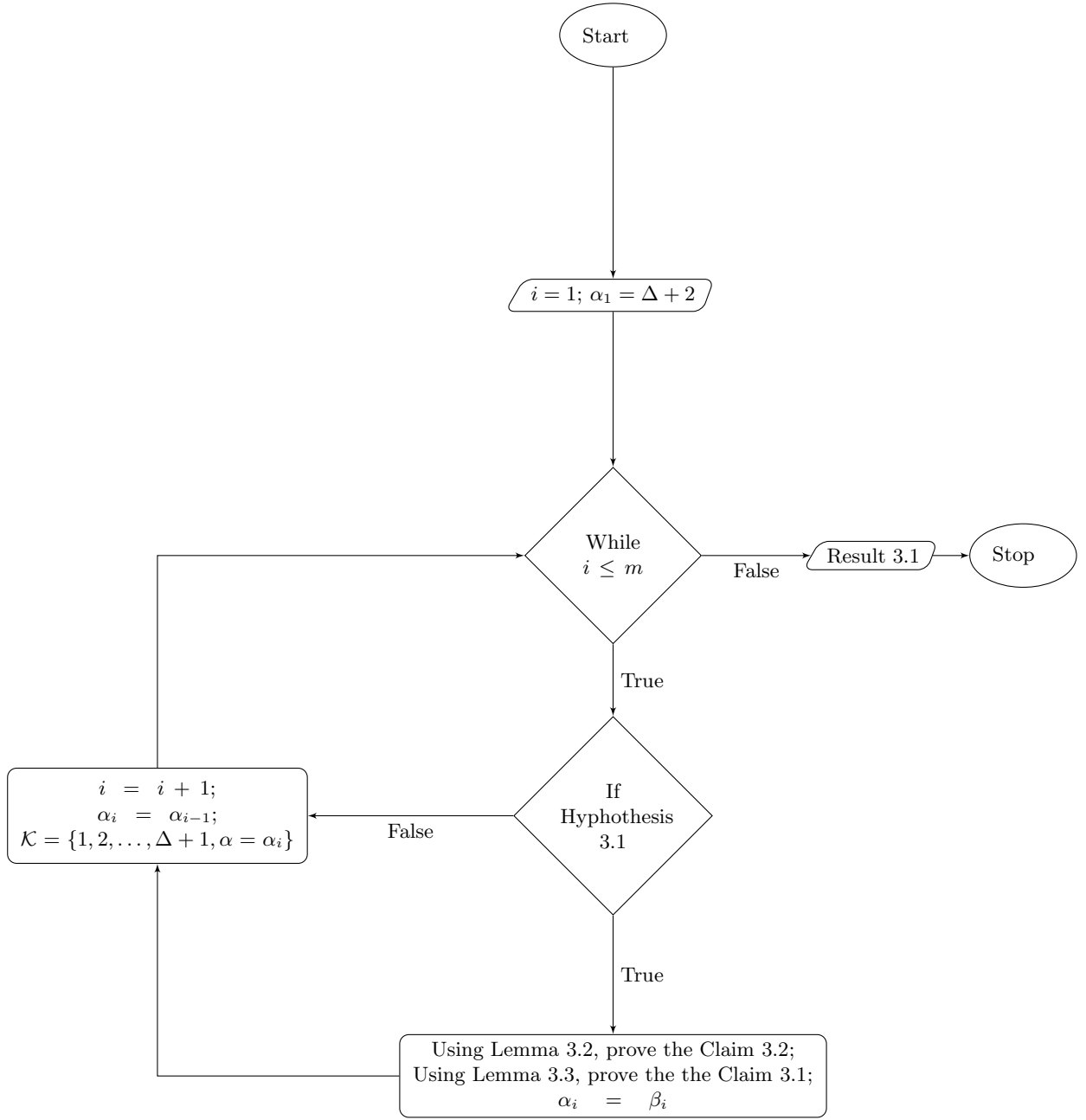
\vspace{1cm}
		\begin{tikzpicture}[node distance = 1.8cm, auto] 
		
		\node [cloud] (init) {Start};  
		\node [io, below of=init](B){$i=1$;\ $\alpha_1 = \Delta+2$} ;
		\node [decision,below of=B](C){While $i \leq m$};
		\node [decision,below of=C](D){If \\Hyphothesis \ref{h1}};
		\node [block, left of = D](E){$i = i+1$;\\ $\alpha_i = \alpha_{i-1}$;\\$\mathcal{K} = \{1, 2, \dots, \Delta+1, \alpha = \alpha_i\}$};
		\node [block1,below of = D](F){Using Lemma \ref{l2}, prove the Claim \ref{cl2};\\ Using Lemma \ref{l3}, prove the the Claim \ref{cl3};\\$\alpha_i=\beta_i$};
		\node [io, right of=C] (J) {Result \ref{cl4}};	
		\node [cloud,  right of=J] (H) {Stop}; 
		
		\path [line] (init) -- (B);   
		\path [line] (B) -- (C); 
		\path [line] (C) -- node[right]{True}(D);
		\path [line] (D) -- node[below]{False}(E);  
		\path [line] (E) |- (C.west);
		\path [line] (D) -- node[right]{True}(F);
		\path [line] (F) -| (E.south);
		\path [line] (C) -- node[below]{False}(J);
		\path [line] (J) -- (H);

		\end{tikzpicture}  
\clearpage
\center{\textbf{The following schematic diagram gives us the overall outer view of presentation of the paper}}\\

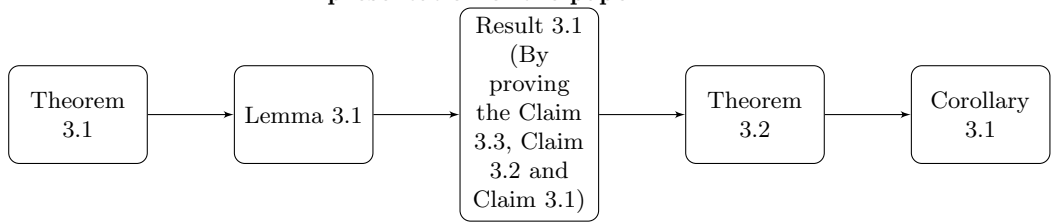
\begin{tikzpicture}[node distance = 0.8cm, auto] 

\node [block2](B1){Theorem \ref{t1}};
\node [block2, right of=B1](C1){Lemma \ref{l1}};
\node [block2, right of=C1](D1){Result \ref{cl4}\\
	(By proving the Claim \ref{cl1}, Claim \ref{cl2} and Claim \ref{cl3})};
\node [block2, right of=D1](E1){Theorem \ref{pt4}};
\node [block2, right of=E1](F1){Corollary \ref{pc1}};

\path [line] (B1) -- (C1); 
\path [line] (C1) -- (D1);
\path [line] (D1) -- (E1);  
\path [line] (E1) -- (F1);

\end{tikzpicture}

\vspace{1cm}
\center{\textbf{Acknowledgments:}}\\

\ \ The author would like to thank Raghu Menon for enthusiastically copy-editing initial drafts and Ryan Alweiss for his inputs that resulted in major changes in the paper. \\

\bibliographystyle{amsplain}

\end{document}